\documentclass[11pt,a4paper]{amsart}
\usepackage{amscd}
\usepackage[T1]{fontenc}
\usepackage[latin1]{inputenc}
\usepackage{enumerate}
\usepackage{cancel}
\usepackage{amsfonts}
\usepackage{mathrsfs}
\usepackage{amsthm}
\usepackage{amssymb}
\usepackage{dsfont}
\usepackage[english]{babel}
\usepackage{amsmath,amscd}
\usepackage{amsmath}
\usepackage{hyperref}
\hypersetup{
    colorlinks,
    linkcolor={red!50!black},
    citecolor={blue!50!black},
    urlcolor={blue!80!black}
}

\usepackage[pdftex]{graphicx}
\usepackage{lmodern}

\usepackage[bottom=3cm, top=3cm, left=3.5cm, right=3.5cm]{geometry}

\usepackage{mathtools}
\mathtoolsset{showonlyrefs,showmanualtags} 

\usepackage[foot]{amsaddr}
\usepackage[alphabetic, abbrev]{amsrefs}

\usepackage{enumitem}
\setlist[itemize]{leftmargin=1cm}

\usepackage{xcolor}

\theoremstyle{plain}
\newtheorem{theo}{Theorem}

\newtheorem{cor}[theo]{Corollary}
\newtheorem{lem}[theo]{Lemma}
\newtheorem{prop}[theo]{Proposition}

\theoremstyle{remark}
\newtheorem{rem}[theo]{Remark}

\theoremstyle{definition} 
\newtheorem{de}[theo]{Definition}

\DeclareMathOperator{\grad}{\mathrm{grad}}

\newcommand{\M}{M}
\newcommand{\T}{\mathsf{T}}
\newcommand{\SR}{\mathit{SR}}
\newcommand{\distr}{\Delta}
\newcommand{\interv}{I}
\newcommand{\eps}{\varepsilon}
\newcommand{\hc}{\zeta}
\newcommand{\geo}{\gamma}

\newcommand{\rd}{\varrho}
\newcommand{\rvf}{\Gamma}

\author{Davide Barilari} 
\email{barilari@math.unipd.it}
\address{Dipartimento di Matematica ``Tullio Levi-Civita'', Universit\'a degli Studi di Padova, Padova, Italy}
\author{Mathieu Kohli}
\email{mathieu.kohli@orange.fr}
\address{CMAP, \'Ecole Polytechnique, Palaiseau, France}
\date{\today}
\title[Geodesic curvature in 3D sub-Riemannian geometry]{On sub-Riemannian geodesic curvature \\ in dimension three}

\begin{document}
\begin{abstract} We introduce a notion of geodesic curvature $k_{\hc}$ along a smooth horizontal curve $\hc$ in a three-dimensional contact sub-Riemannian manifold, measuring how much a horizontal curve is far from being a geodesic. 
We show that the geodesic curvature appears as the first corrective term in the Taylor expansion of the sub-Riemannian distance between two  points on a unit speed horizontal curve
\begin{equation*} 
d_{\SR}^2( \hc(t),\hc(t+\eps))=\eps^2-\frac{k_{\hc}^2(t)}{720} \eps^6 +o(\eps^{6}).
\end{equation*}
The sub-Riemannian distance is not smooth on the diagonal, hence the result contains the existence of such an asymptotics. This can be seen as a higher-order differentiability property of the sub-Riemannian distance along smooth horizontal curves. It generalizes the previously known results on the Heisenberg group.
\end{abstract}
\maketitle
\tableofcontents

\section{Introduction}


It is a classical result in Riemannian geometry that a smooth curve $\hc:I\to M$ on a Riemannian manifold $(M,g)$ is a geodesic if and only if $\nabla^{g}_{\dot \hc} \dot \hc=0$, where $\nabla^{g}$ denotes the Levi-Civita connection. 
Indeed, if a smooth curve $\hc:I\to M$ is parametrized by length, the quantity $$\kappa^{g}_{\hc}=\|\nabla^{g}_{\dot \hc} \dot \hc\|$$ is called the \emph{geodesic curvature} of  $\hc$ and quantifies how much  the curve $\hc$ is far from being a geodesic. More precisely, let us denote by $d$ the Riemannian distance on $(M,g)$. A smooth curve  $\hc:I\rightarrow \M $ parametrized by arc length is a geodesic if and only if it satisfies $d^2( \hc(t),\hc(t+\eps))=\eps^2$ for all $t$ and for $\eps>0$ small enough. Indeed one can prove the following asymptotic expansion: for every $t\in I$, when $\eps\to 0$
\begin{equation} \label{eq:mainr}
d^2( \hc(t),\hc(t+\eps))=\eps^2-\frac{\kappa^{g}_{\hc}(t)^{2}}{12} \eps^4 +o(\eps^{4}).
\end{equation}
This formula provides a purely metric interpretation of the geodesic curvature and actually could serve as a definition of the latter.

It is natural to ask whether a similar result holds in the setting of sub-Riemannian geometry, where a notion of geodesic curvature has been up to now only investigated, to our best knowledge, in the case of the Heisenberg group. In \cite{diniz2016gauss}, \cite{balogh2017intrinsic} a notion of geodesic curvature has been introduced in the context of a sub-Riemannian Gauss-Bonnet-like theorem. In \cite{chiu2,chiu1} the authors find complete invariants for regular curves (not only horizontal) in the Heisenberg groups. Finally, in \cite{MK19}, the second named author proves that a formula analogue to \eqref{eq:mainr} holds in the sub-Riemannian Heisenberg group, using an explicit expression for the sub-Riemannian distance. An interesting remark is that the correction one finds is at higher order with respect to the Riemannian case. 
 
The main goal of this paper is to prove the analogue of expansion \eqref{eq:mainr} in 3D contact sub-Riemannian geometry, stated in Theorem~\ref{t:main} below for \emph{smooth} horizontal curves. We stress that the existence of such an asymptotic is a priori not guaranteed, since the sub-Riemannian distance is not smooth on the diagonal. Let us mention that the study of higher-order asymptotics of the sub-Riemannian distance along geodesics have been performed in \cite{agrachev2018curvature} to extract a notion of ``sectional curvature'' of a sub-Riemannian manifold.



\subsection{Statements of the results}
Let $M$ be a three-dimensional (3D) contact sub-Riemannian manifold and let $\hc:I\to M$ be a smooth horizontal curve parametrized by arc length. In what follows $I$ always denotes an open interval containing zero.

We associate with every smooth horizontal curve $\hc$ its characteristic deviation $h_\hc : I\rightarrow \mathbb{R}$. In terms of the Tanno (or Tanaka-Webster) connection $\nabla$, we set: 
\begin{align} \label{eq:hzetag}
 h_{\hc}&=g(\nabla_{\dot \hc}\dot \hc,J\dot \hc),
\end{align}
where $J$ is the almost complex structure defined on the distribution by the contact structure (we refer to Section~\ref{s:prel} for precise definitions).

The characteristic deviation characterizes horizontal curves with given initial position and velocity.
\begin{prop}\label{p:char} \label{p:1}Let $\M$ be a 3D contact sub-Riemannian structure.
If $\hc_1,\hc_2: I\rightarrow \M $ are two smooth horizontal curves parametrized by arc length such that:
\begin{itemize}
\item[(i)] $\hc_1(0)=\hc_2(0)$ and $\dot \hc_1(0)=\dot \hc_2(0)$,
\item[(ii)] $h_{\hc_1}(t)=h_{\hc_2}(t)$ for every $t\in I$.
\end{itemize}
then $\hc_1(t)=\hc_2(t)$ for every $t\in I$.
\end{prop}
A particular case of three-dimensional contact sub-Riemannian structure $M$ is given by isoperimetric problems over a two-dimensional Riemannian manifold $N$ (cf.\ Section~\ref{s:isoper} for more details). In this case every horizontal curve $\hc$ on $M$ is the lift of a unique curve on the underlying Riemannian manifold $N$ and the characteristic deviation is the geodesic curvature of $\pi_{N}(\hc)$, where $\pi_{N}:M\to N$ is the projection. In this case Proposition~\ref{p:char} has a clear geometric interpretation. 

The characteristic deviation is not the correct quantity describing whether a curve is a sub-Riemannian geodesic or not. 
Recall that a  horizontal curve $\hc:I\rightarrow \M $ is a \emph{geodesic} if short arcs of $\hc$ realize the sub-Riemannian distance between its endpoints. The geodesic curvature $k_\hc : I\rightarrow \mathbb{R}$ of a horizontal curve $\hc:I\rightarrow \M $ parametrized by length is defined by
$$ k_{\hc}=  \frac{d}{dt}g(\nabla_{\dot \hc}\dot \hc,J\dot \hc) -g(T(X_{0},\dot \hc),\dot \hc),$$
where $X_{0}$ is the Reeb vector field associated with the contact structure, and $T$ is the torsion associated to $\nabla$ (we refer to Section~\ref{s:prel} for precise definitions).
\begin{prop}\label{propcharacterizationgeod}
Let $\M$ be a three-dimensional contact sub-Riemannian  structure and let $\hc:I\rightarrow \M $ be a smooth horizontal curve parametrized by arc length. Then $\hc$ is a geodesic if and only if $k_{\hc}(t)=0 $ for every $t\in I$. 
\end{prop}
For a given geodesic $\hc:I\rightarrow \M $ and every $t\in I$ one has $d^2( \hc(t),\hc(t+\eps))=\eps^2$ for $\eps>0$ small enough.
The main result of this paper is to show that the geodesic curvature appears as a corrective term.
\begin{theo} \label{t:main}
Let $\M$ be a three-dimensional contact sub-Riemannian  structure and let $\hc:I\rightarrow \M $ be a smooth horizontal curve parametrized by arc length. Then for every fixed $t\in I$, we have the expansion for $\eps\to 0$
\begin{equation} \label{eq:ordine6}
d_{\SR}^2( \hc(t),\hc(t+\eps))=\eps^2-\frac{k_{\hc}^2(t)}{720} \eps^6 +o(\eps^{6}).
\end{equation}
\end{theo}
Notice the qualitative behavior  with respect to what happens in the Riemannian case in \eqref{eq:mainr}. Indeed, in the Taylor expansion the corrective term due to the curvature appears at order 6 while in the Riemannian case it appears at order 4 (cf.\ also \cite{MK19} for the case of the Heisenberg group).

Some preliminary investigations in the higher-dimensional contact case are discussed in \cite{MKphd}, showing that the fact that the first corrective term appearing in \eqref{eq:ordine6} is at order 6 is related to the distribution being two-dimensional.

\subsection{Strategy of the proof of Theorem~\ref{t:main}} 

Fix $p\in M$ and consider the cut locus $\mathrm{cut}(p)\subset M$, the set of points where the function $d_{\SR}^{2}(p,\cdot)$ is not smooth. We have $p\in \mathrm{cut}(p)$ but  $M\setminus \mathrm{cut}(p)$ is an open dense set in $M$. We refer to Section~\ref{s:distcut} for more details on the cut locus. 
On $M\setminus \mathrm{cut}(p)$ it is well-defined the radial vector field $\rvf$ as follows: for $q\in M\setminus \mathrm{cut}(p)$ the vector $\rvf(q)$ is the tangent vector at $q$ to the geodesic joining $p$ and $q$ in time 1.

Consider now $\hc :I\rightarrow M$  a smooth horizontal curve parametrized by arc length on an open interval $I$ containing $0$ and such that $\hc(0)=p$. For all $t\in I$ such that $\hc(t)\notin \mathrm{cut}(p)$ we can decompose the velocity vector of $\hc$ as follows
\begin{equation}\label{deftheta}
\hc'(t)=\cos( \theta(t) )\rvf( {\hc}(t) )  + \sin( \theta(t) ) J\rvf ( {\hc}(t) ).
\end{equation}
where $\theta$ takes values in $S^1$ and is defined on a subset of $I\setminus \{0\}$. 

We stress that the vector fields $\Gamma$ and $J\Gamma$ are singular at $p$, hence the regularity of the function $t\mapsto \theta(t)$ at $t=0$ is not guaranteed.

The proof consists in the following facts:
\begin{itemize}
\item[(a)]  $\hc(t)\notin \mathrm{cut}(p)$ for all $t\in I\setminus \{0\}$ (by restricting $I$, if needed)
\item[(b)] the function $\theta:I\setminus \{0\}\to S^{1}$ can be extended to $I$ continuously with $\theta (0)=0$
\item[(c)]  such a function $\theta:I \to S^{1}$ is of class $C^{2}$ in a neighborhood of $0$
\end{itemize}

With the $C^{2}$ regularity of the function $\theta$, it  is not difficult to obtain a Taylor expansion of the distance along the curve $\hc$ in terms of the derivatives of $\theta$. The final step is to recover the geometric invariant in terms of of $\theta$.


\subsection{Structure of the paper} After some preliminaries in Section 2,  in Section 3 and 4 we introduce characteristic deviation and geodesic curvature in 3D contact SR geometry. 
Sections 6 and 7 contain the main results of the paper. The Appendix contains some technical lemmas and a self-contained discussion on Jacobi fields.
\subsection*{Acknowledgments} The authors wish to thank Luca Rizzi for useful discussions.
This work was supported by the Grant ANR-15-CE40-0018 ``Sub-Riemann\-ian Geometry and Interactions'' of the French ANR.

\section{Preliminaries} \label{s:prel}

All along the paper,  $\M$  denotes a three dimensional smooth manifold. We endow $\M$ with  a contact form $\omega$, which is a smooth differential $1$-form such that $\omega\wedge d\omega \neq 0$.  

We define the \emph{distribution} $\distr:=\ker \omega$. Notice that $d\omega|_{\distr}$ is non-degenerate.  The choice of a smooth metric $g$ on $\distr$ defines a contact sub-Riemannian structure on $\M$. We denote by $\| \cdot \|$ the norm on $\distr$ associated with the metric $g$.

Notice that if $f$ is a smooth non-vanishing scalar function $f$, then $f\omega$ is also a contact form and $\ker (f\omega)=\distr$. Thus, it is not restrictive  to assume that $\omega$ is chosen in such a way that $d\omega|_{\distr}=-\mathrm{vol}_{\distr}$, where $\mathrm{vol}$ is the volume form associated with the metric $g$ on $\distr$. We assume this normalization is fixed in what follows.\footnote{The minus sign here is a convention. It is chosen in such a way that the Reeb vector field satisfies $[X_{1},X_{2}]=X_{0} \mod \distr$ for any directed orthonormal frame $\{X_{1},X_{2}\}$ of the distribution.}


%

A Lipschitz curve $\hc : I\rightarrow \M$ is said to be \emph{horizontal} if $\dot \hc(t)\in\distr_{\hc(t)}$ for a.e.\ $t\in I$. The length of a horizontal curve $\hc:I\rightarrow \M$ is given by 
\begin{align*}
\ell_{\SR}(\hc)=\int_{I}\| \dot \hc(s) \| \text{d}s.
\end{align*}
The sub-Riemannian distance is then defined as follows
\begin{align*}
d_{\SR}(x,y)=\inf \{\ell_{\SR}(\hc)\mid \hc :I\rightarrow \M \text{ horizontal curve joining $x$ and $y$}\}.
\end{align*}
Since the distribution is contact, $\distr$ is bracket-generating and every pair of points on a connected component of $\M$ can be joined by a horizontal curve of finite length. For more detailed introduction to sub-Riemannian geometry we refer to \cite{ABB}. 

Given a contact form $\omega$ on $\M$, the Reeb vector field $X_0$ is defined as the unique vector field satisfying $\omega(X_{0})=1$ and $d\omega(X_{0},\cdot)=0$.
Notice that $X_0$ is transverse to the distribution. We still denote by $g$ the unique extension of the sub-Riemannian metric to a Riemannian one such that $X_{0}$ is orthogonal to $\distr$ and of norm one.

The map
$
J:\distr \rightarrow \distr
$
is the linear endomorphism on $\distr$ that satisfies
\begin{align}\label{eq:gJ}
g( X,JY ) =\text{d}\omega ( X,Y ),
\end{align}
for every horizontal vector fields $X$ and $Y$.
Notice that $g(X,JX)=0$ for every $X$. Indeed $\{X,JX\}$ is a direct orthonormal frame for $\distr$, for every unit vector $X$ in $\distr$. The map $J$ can be  extended to $TM$ compatibly with \eqref{eq:gJ} by setting $J(X_{0})=0$.



Given any orthonormal frame of the distribution $\{ X_1 , X_2\}$, the triple $\{ X_0,X_1,X_2\}$ is a frame of $T\M$. It is easy to see that due to our normalization choice we have
\begin{align}\label{decompositionLieijk}
[ X_i, X_j ] =\sum_{k=0}^2 c_{i j}^{k} X_k.
\end{align}
where $c_{ij}^{k}$ are smooth functions satisfying
$c_{0i}^{0}=0$ for every $i=0,1,2$ and $c_{12}^{0}=1$.
\begin{rem} \label{rem:not}
Given three vector fields $X,Y,Z$, we set $c_{X,Y}^{Z}:=g([X,Y],Z)$. Notice that this is compatible with the above notation in the following sense: for $i,j,k=0,1,2$ we have $c_{X_{i},X_{j}}^{X_{k}}=c_{i  j}^{k}$.
\end{rem}

\section{The characteristic deviation}

In this section we introduce the characteristic deviation and we prove that smooth horizontal curves with given initial base point and velocity are uniquely characterized by their characteristic deviation, namely Proposition~\ref{p:1}.

We start by introducing the canonical connection on any 3D contact manifold, which is called Tanno connection \cite{tanno89}.
\begin{de}
 The Tanno connection
$
  \nabla$
 is the only linear connection on $M$ satisfying
\begin{itemize}
\item[(i)] $ \nabla \omega =0$,
\item[(ii)] $ \nabla X_0 =0$,
\item[(iii)] $\nabla g=0$,
\item[(iv)] $T(X,Y)=d\omega( X,Y)X_0$,  for $X,Y$ horizontal vector fields,
\item[(v)] $T( X_0, JX) = -JT( X_0,X )$,  for $X$ any vector field,
\end{itemize}
where $T$ denotes the torsion associated with the connection $\nabla$.
\end{de}
Notice that on 3D contact manifolds one automatically has $\nabla_{X}(JY)=J\nabla_{X}Y$, i.e., the contact structure is CR. Hence Tanno connection $\nabla$ coincides with Tanaka-Webster connection in this case (we refer to \cite{agrachev2017sub} for a discussion in a similar notation).
\begin{rem} \label{r:trace0} From the properties of the Tanno connection it is easy to show that for every horizontal vector field $X$, 
 one has $g([X_{0},X],X)=-g([X_{0},JX],JX)$. 

\smallskip
More precisely, one can show that $\tau(X):=T(X_{0},X)$ is horizontal, for every horizontal $X$. The map $\tau:\distr\to \distr$ is symmetric with respect to the sub-Riemannian inner product and has zero trace. 
\end{rem}
\subsection{The characteristic deviation}
 Let us consider a smooth horizontal curve
$
  \hc :\interv\rightarrow \M
$
parametrized by arc length. We say that a smooth vector field $\T $ is an extension of the velocity field of $\hc$ if for every $t\in\interv$ :
\begin{align*}
 \T ( \hc(t) )=\dot \hc(t).
\end{align*}

Notice that, since the vector field $\T $ is horizontal and  $\nabla \omega=0$, then $\nabla_{\T }\T $ is horizontal. Moreover $\nabla_{\T }\T $ it orthogonal to $\T $ since
 \begin{align*}
 g( \T ,\nabla_{\T }\T  ) =\frac{1}{2}\T g( \T ,\T  ) =0.
 \end{align*}
This permits to introduce the following definition. 
\begin{de}
\label{eqdiffTannoh0}
 Let $\hc :\interv\rightarrow \M$ be a smooth horizontal curve  parametrized by arc length and let $\T $ be a smooth horizontal and normalized vector field extending the velocity field of $\hc$. We define the \emph{characteristic deviation} $h_{\hc} :\interv\rightarrow \mathbb{R}$ by
 \begin{align}\label{eq:hc22}
\nabla_{\T }\T ( \hc(t) ) &=h_{\hc}(t)J\T ( \hc(t) ),\qquad \forall t\in \interv.
 \end{align}
\end{de}
It follows from the definition that this quantity does not depend on the extension $\T$. The choice of an extension is indeed not even necessary to write down equation \eqref{eq:hc22}, since one can define
 \begin{align}
\nabla_{\dot \hc(t) }\dot \hc(t)  &=h_{\hc}(t)J\dot\hc(t) ,\qquad \forall t\in \interv.
 \end{align}
It will be convenient in what follows to use the notation \eqref{eq:hc22} to work with vector fields defined on $M$, and remember that the characteristic deviation is independent on the extension.
\begin{lem}
The characteristic deviation can be expressed as follows
\begin{align} \label{eq:hzetag}
 h_{\hc}(t)&=g_{\hc(t)}([J\T ,\T ],\T )=-c_{\T ,J\T }^{\T }( \hc(t) ).
\end{align}
where the last identity is understood in the sense of Remark~\ref{rem:not}.
\end{lem}
\begin{proof} 
Using the properties of the connection $\nabla$, we can compute 
 \begin{align*}
 -c_{\T ,J\T }^{\T } &=-g( [\T ,J\T ], \T  )=-g( \nabla_{\T }J\T  -\nabla_{J\T }\T , \T  )\\
 &=-g( \nabla_{\T }J\T  , \T  ) +\frac{1}{2}J\T g( \T ,\T  )
 =g( J\T ,\nabla_{\T }\T   )
 \end{align*}
 where we used that the torsion of horizontal vector fields is vertical and the identities $\nabla_{\T }J\T =J\nabla_{\T} \T$, $J^{2}=-1$ and $J\T g( \T ,\T  )=0$. 
\end{proof}
The characteristic deviation can be easily computed if one decomposes the tangent vector to $\hc$ along an orthonormal frame. 
\begin{lem} \label{lemma10} Let $\hc :I\to  M $
be a smooth horizontal curve parametrized by arc length and let us write
\begin{equation}\label{nuova1}
\T=\cos( \theta )X_{1} +\sin( \theta )X_{2}.
\end{equation}
with respect to some orthonormal frame $X_{1},X_{2}$. Then along $\hc(t)$ we have the identity
\begin{equation} \label{eq:hzetath}
 h_{\hc}=\dot \theta -c_{12}^{1}\cos \theta -c_{12}^{2}\sin \theta.
\end{equation}
where $c_{ij}^{k}$ are defined as in \eqref{decompositionLieijk} and $\dot \theta$ denotes the derivative of $\theta$ along $\hc$.
\end{lem}

The proof is a direct computation which makes use of the following observation: if $X,Y$ are two smooth horizontal normalized vector fields such that
 \begin{align*}
Y=\cos( \psi )X +\sin( \psi )JX,
\end{align*}
for some smooth function
$
  \psi :\M\rightarrow S^{1}
$, then
\begin{align}\label{nuova2}
 [ Y,JY ]&=[ X,J X ] -\grad_{\SR}(\psi).
\end{align}
Here we denote by $\grad \psi$ the \emph{horizontal gradient}, i.e., the horizontal vector field such that 
$ \text{d}\psi(X) =g(\grad_{\SR}(\psi) ,X )$ for any  horizontal $X$.
It is easy to check that
 \begin{align}\label{expressiongrandient}
 \grad_{\SR}(\psi)=( X_1\psi )X_1 +( X_2 \psi )X_2.
 \end{align}
    for any orthonormal frame $ X_1,X_2 $ of the distribution.
\begin{proof}[Proof of Lemma~\ref{lemma10}]
Using \eqref{nuova1} and \eqref{nuova2} we have that
\[
[\T,J\T]=[X_{1},X_{2}]-\grad_{\SR}(\theta)=(c_{12}^{1}X_{1}+c_{12}^{2}X_{2}+X_{0})-(X_{1}\theta) X_{1}-(X_{2}\theta) X_{2}
\]
hence by definition of characteristic deviation \eqref{eq:hzetag} we have
\[
h_{\hc}=-g([\T,J\T],\T)=-(c_{12}^{1}\cos \theta+c_{12}^{2}\sin \theta)+(X_{1}\theta) \cos \theta+(X_{2}\theta) \sin \theta
\]
and formula \eqref{eq:hzetath} follows.
\end{proof}

\subsection{Two directional invariants}
We denote by $S\M$ the spherical horizontal bundle, i.e., the set of unit vectors in the distribution. Let us define the two tensors $\eta,\iota : SM \rightarrow \mathbb{R}$ by
\begin{align} \label{eq:etaiota}
 \eta( X )=g(\tau(X),X),\qquad 
  \iota( X)=g(\tau(X),JX).
\end{align}

\begin{lem} \label{etaiota} For $\theta\in S^1$ and $X\in SM$,  set 
$
 X_{\theta}:=\cos( \theta )X +\sin( \theta )JX
$.
Then 
\begin{align*}
\eta(X_{\theta})&=\cos(2\theta)\eta(X)+\sin(2\theta) \iota(X),\quad \iota(X_{\theta})=-\sin(2\theta)\eta(X)+\cos(2\theta) \iota(X). 
\end{align*}
\end{lem}

\begin{proof} \label{appelem}
Since $\tau$ is traceless and is symmetric, we have for every $X$ horizontal 
\begin{equation}\label{eq:tracesym}
g(\tau(X),X)+g(\tau(JX),JX)=0,\qquad g(\tau(X),JX) = g(\tau(JX),X).
\end{equation}
Then using bilinearity and \eqref{eq:tracesym} one has
\begin{align*}
\eta(X_{\theta})=g(\tau(X_{\theta}),X_{\theta})
&=\cos^{2}(\theta) g(\tau(X),X)+\sin^{2}(\theta) g(\tau(JX),JX)\\
&\qquad +\cos \theta \sin \theta \left( g(\tau(X),JX) + g(\tau(JX),X) \right)\\
&=\cos(2\theta)g(\tau(X),X)+\sin(2\theta) g(\tau(X),JX) \\
&=\cos(2\theta)\eta(X)+\sin(2\theta) \iota(X) 
\end{align*}
Similarly one obtains
$
\iota(X_{\theta})=-\sin(2\theta)\eta(X)+\cos(2\theta) \iota(X) 
$, proving the lemma.
\end{proof}


\begin{rem} \label{rem:chichi} From Lemma~\ref{etaiota} one can see the following fact: the Reeb vector field $X_{0}$ is a Killing vector field for the sub-Riemannian structure if and only if both (or equivalently, one among) $\eta$ and $\iota$ vanishes for all horizontal unit vectors. 
Indeed 
\begin{align}\label{eq:sumsq}
\|\tau(X)\|^{2}=\eta(X)^2 +\iota(X)^2
\end{align}
Moreover we have $\|\tau(X)\|=\chi(x)$ for all $x$ in $\M$, and $X\in S_{x}M$, where $\chi(x)$ is the local invariant, as discussed in \cite{agrachev2012sub} (see also \cite{agrachev1996exponential}). In particular the right hand side of \eqref{eq:sumsq} is independent of the unit vector $X$ and depends only on the base point $x$. Moreover we stress that $|\eta(X)|\leq \chi(x)$ for every $X$ unitary and based at $x$, hence $\eta$ is locally bounded on $SM$ since $\chi$ is a smooth function in $M$.

\end{rem}
\begin{rem}
Notice that one can rewrite \eqref{eq:etaiota} as
\begin{align} \label{eq:etaiota1}
 \eta( X )=g([X,X_{0}],X),\quad
  \iota( X)=\frac12\left(g([JX,X_{0}],X)+g([X,X_{0}],JX)\right).
\end{align}
\end{rem}

\subsection{Existence and uniqueness: proof of Proposition~\ref{p:1}}
We now show that there exists a unique horizontal curve with assigned characteristic deviation, for a given initial point and velocity. This will prove Proposition \ref{p:1}.
\begin{prop} Let $\M$ be a complete 3D sub-Riemannian contact structure.
Given $x\in M$, a unit vector $v\in \distr_{x}$ and a smooth function
$ \varphi: \interv\to \mathbb{R}$,
there exists a unique smooth horizontal curve $\hc :\interv\to  \M$ parametrized by arc length such that $\hc(0)=x$, $\dot \hc(0)=v$, and $h_\hc(t)=\varphi(t)$ for all $t\in \interv$.
\end{prop}
\begin{proof}

(i).\ 
Let $\hc_1,\hc_{2} :\interv\to\M$ be two smooth horizontal curves parametrized by arc length such that $\hc_{i}(0)=x$, $\dot \hc_{i}(0)=v$ for $i=1,2$, and such that
$ h_{\hc_1}=\varphi=h_{\hc_2}$. It follows that $\hc_1 ,\hc_2$  both satisfy the same differential equation
 \begin{align*}
 \nabla_{\dot \hc}\dot \hc &=\varphi(t)J\dot \hc. 
 \end{align*}
having the same initial conditions. Hence 
$ \hc_1=\hc_2$.
 
(ii).\ Fix $x\in M$, a unit vector $v\in T_{x}M$, and a smooth function
$ \varphi: \interv\to \mathbb{R}$. Since $\M$ is complete, there exists $\hc: \interv\rightarrow \M$ a smooth solution to the Cauchy problem:
 \begin{align*}
  \nabla_{\dot \hc} \dot \hc=\varphi(t) J\dot \hc,\quad \hc(0)=x,\dot \hc(0)=v,
 \end{align*}
We are left to show that $\hc$ is horizontal and  hasunit speed. 
By definition of the Tanno connection, $\nabla \omega =0$, hence
\begin{align*}
 \frac{\text{d}}{\text{d}t} \omega( \dot \hc(t) ) =\omega( \nabla_{\dot \hc(t)}  \dot \hc) = \omega( \varphi(t)J\dot \hc(t) )= 0,
\end{align*}
which implies that $\dot \hc(t)$ is horizontal for any $t \in \interv$. Moreover  $\nabla g=0 $, hence
\begin{align*}
 \frac{\text{d}}{\text{d}t}g ( \dot \hc(t),\dot \hc(t) )&= 2 g ( \nabla_{\dot \hc(t)} \dot \hc, \dot \hc(t) )= 2 g ( \varphi(t)J\dot \hc(t), \dot \hc(t) ) =0,
\end{align*}
which means that $\dot \hc(t)$ is  a unit vector for every $t\in \interv$, since it has norm 1 for $t=0$.
\end{proof}
We end this section by showing that for a sub-Riemannian problem that is defined by a so-called isoperimetric problem on a Riemannian surface $N$ (cf.\ \cite{AG99} for the terminology), the characteristic deviation of a horizontal curve is the geodesic curvature of its projection.
\subsection{Isoperimetric problems}\label{s:isoper}
Let $(N,g_{N})$ be a two-dimensional Riemannian manifold and $A$ be a 1-form on $N$. For $x,y\in N$ let 
\[\Omega^{N}_{x,y}=\{\alpha:[0,T]\to N\mid  \alpha\in C^{\infty}, \alpha(0)=x, \alpha(T)=y\}.
\]
The \emph{isoperimetric problem} on $M$ associated with the 1-form $A$, is the following
\begin{equation} \label{eq:iso1}
\inf\left\{\ell(\alpha) \mid \alpha\in \Omega^{N}_{x,y}, \int_{\alpha}A=c\right\},
\end{equation}
where $c$ is a real constant and $x,y$ are points on $N$. If one chooses $A$ in such a way that $dA=\mathrm{vol}_{N}$ then one recovers the classical problem of minimizing the length of a curve spanning a fixed area. 

One can introduce the sub-Riemannian structure on $M=N\times \mathbb{R}$ by lifting a curve $\alpha$ on $N$ to a curve $\hc(t)=(\alpha(t),z(t))$ where
\[
z(t)=\int_{0}^{t}A(\dot \alpha (s))ds.
\]
The lifted curves $\hc$ are then tangent to the distribution defined as $\distr=\ker \omega$ where $\omega=dz-A$. Notice that $\omega$ is contact if and only if $dA$ is never vanishing on $N$. If $\pi:M\to N$ denotes the canonical projection, then $\pi_{*}$ restricts to an isomorphism between $\distr$ and $TN$. Denoting by $g=\pi^{*}g_{N}$ the pull-back of the metric of $N$ on the distribution $\distr$,  problem \eqref{eq:iso1} rewrites as
\begin{equation}\label{eq:iso2}
\inf\{\ell_{\SR}(\hc) \mid \hc \ \mathrm{horizontal},\, \hc(0)=(x,0),\,\hc(T)=(y,c)\}.
\end{equation}
\begin{prop} 
Let $\hc:[0,T]\to M$ be the smooth horizontal lift of a smooth curve $\alpha:[0,T]\to N$. Then $h_{\hc}(t)=\kappa^{N}_{\alpha}(t)$, where $\kappa^{N}_{\alpha}$ is the Riemannian geodesic curvature of $\alpha$ on $N$.
\end{prop}
\begin{proof} Fix an orthonormal basis $X_{1},X_{2}$ for the distribution and write 
$\dot \hc=\cos (\theta) X_{1}+\sin (\theta) X_{2}$.
Then it is easy to see that
$\dot \alpha=\cos (\theta) Y_{1}+\sin (\theta) Y_{2}$,
where $Y_{i}:=\pi_{*}X_{i}$ is an orthonormal basis for the Riemannian metric on $N$. Then it is a standard computation to show that
\[
\kappa^{N}_{\alpha}=\dot \theta -c_{12}^{1}\cos \theta -c_{12}^{2}\sin \theta.
\]
is the geodesic curvature of $\alpha$ on $N$. The proof is completed by Lemma~\ref{lemma10}.
\end{proof}

\subsection{Normal coordinates}
We express the characteristic deviation in a particular adapted set of coordinates called normal coordinates, as introduced in \cite{el1996small}.
\begin{prop}\label{defGauthier} If $p$ is a point in $\M$, there exist a neighborhood $U$ of $p$ and coordinates $( x,y,z )$ on $U$, and smooth functions
$u,v :U\rightarrow \mathbb{R}$
that satisfy
\begin{align*}
u( 0,0,z ) =v( 0,0,z ) =\frac{\partial v}{\partial x}( 0,0,z )=\frac{\partial v}{\partial y}( 0,0,z ) =0, 
\end{align*}
such that the two vector fields
\begin{align*}
X_1 &=\left( \frac{\partial}{\partial x}-\frac{y}{2}\frac{\partial}{\partial z} \right) +uy\left( y\frac{\partial}{\partial x} -x\frac{\partial}{\partial y} \right) -v\frac{y}{2}\frac{\partial}{\partial z},\\
X_2&=\left(\frac{\partial}{\partial y} +\frac{x}{2}\frac{\partial}{\partial z}\right) -ux\left( y\frac{\partial}{\partial x} -x\frac{\partial}{\partial y} \right) +v\frac{x}{2}\frac{\partial}{\partial z},
\end{align*}
define an orthonormal frame of the distribution around $p=(0,0,0)$.
\end{prop}
It can be easily checked from the definition of Reeb vector field $X_{0}$ that in normal coordinates
\begin{equation}\label{Reebz} 
[X_{1},X_{2}](0)=X_{0}(0)=\tfrac{\partial}{\partial z}.
\end{equation}

In normal coordinates the characteristic deviation of a horizontal curve parametrized by arc length leaving from the origin is computed explicitly.
\begin{prop}\label{propcharacteristicdevGauthier} Let us consider $(x,y,z)$ a system of normal coordinates around $p\in \M$.
If $\hc :\interv \to \M$ is a smooth horizontal curve such that $\hc(0)=p$ and parametrized by arc length with $\hc(t)=(x(t),y(t),z(t))$
\begin{align*}
\dot z(0)=\ddot z(0)=0 ,\quad 
z^{(3)}(0)=\frac{h_{\hc}(0)}{2}.
\end{align*}
In particular we have
\begin{align*}
 h_{\hc}(0)&=\lim_{t\rightarrow 0}\frac{12 z(t) }{( x^2(t) +y^2(t) )^{3/2}}.
\end{align*}
\end{prop}
\begin{proof}
Let us consider $\T $ a smooth horizontal unit  vector field extending the velocity field of $\hc$ and 
such that $
\T =\cos( \psi )X_1 + \sin( \psi )X_2$. 
We have,  using   \eqref{Reebz}:
\begin{align}\label{eq:derzz}
h_{\hc}(0)&=-g([ \T ,J\T  ]( \hc(0) ),\T ( \hc(0) )  ) \nonumber \\
&=-g([ X_1,X_2 ] ( \hc(0) )-\grad_{\SR}(\psi)( \hc(0) ),\T ( \hc(0) )  )=\frac{\text{d}}{\text{d}t}\bigg|_{t=0}\psi( \hc(t) ).
\end{align}
Moreover, denoting $\psi(t):=\psi(\hc(t))$ (and similarly for $u(t),v(t)$), we have for $t\in I$
\begin{align*}
\dot {x}(t)&=\cos( \psi(t) )( 1+u(t)  y(t)^2)-\sin( \psi(t)) u(t)  x(t)y(t),\\
\dot {y}(t)&=\sin( \psi(t) )( 1+u(t) x(t)^2)-\cos(\psi(t) ) u(t)  x(t)y(t),\\
\dot {z}(t)&=\left( -\cos( \psi(t) )\frac{y(t)}{2} + \sin( \psi(t) )\frac{x(t)}{2} \right)( 1+v(t)).
 \end{align*}
By Taylor expansion one has
\begin{align}\label{eq:xy00}
 x(t)&=\cos(\psi(0))t+o(t),\quad y(t)=\sin(\psi(0))t+o(t),
 \end{align}
 After some  computations, using \eqref{eq:derzz} and the properties of the function $v$, one has
 \begin{align*}
 \dot z(0)=\ddot z(0)=0 ,\quad 
z^{(3)}(0)=\frac12 \frac{\text{d}}{\text{d}t}\bigg|_{t=0}\psi( \hc(t) )=\frac{h_{\hc}(0)}{2}.
\end{align*}
which combined with \eqref{eq:xy00}  implies  the conclusion.
\end{proof}

%

\section{The geodesic curvature}

In this section we prove Proposition~\ref{propcharacterizationgeod}, which states that the geodesic curvature $k_{\hc}$ along a horizontal curve $\hc$ is identically zero if and only if the curve is a geodesic. 

\medskip
Let us start by introducing the geodesic curvature of  a smooth horizontal curve.

\begin{de} \label{defgeocurv} Let
$
\hc :\interv\rightarrow\M
$
be a smooth horizontal curve parametrized by arc length. The \emph{geodesic curvature} of $\hc$ is the smooth function $k_{\hc} : \interv  \to \mathbb{R}$ defined by
\[
 k_{\hc}(t):= \frac{d}{dt} h_{\hc}(t) -\eta ( \dot \hc(t) ).
\]
\end{de}
Here $h_{\hc}$ is the characteristic deviation and  $\eta$ is the directional invariant introduced in Proposition~\ref{etaiota}.
Notice that in the language of tensors $ k_{\hc}$ can be rewritten as 
$$ k_{\hc}=  \frac{d}{dt}g(\nabla_{\dot \hc}\dot \hc,J\dot \hc) -g(\tau(\dot \hc),\dot \hc).$$

\smallskip
Proposition~\ref{propcharacterizationgeod}, stated in a different but equivalent form, already appeared in the literature, see 
\cite[Proposition 15]{rumin}  (see also \cite[Lemma 1.1]{chan}). 

Here we give a self-contained proof following our notation, based on the Pontryagin maximum principle, which states that on a contact manifold all geodesics are projections of solutions of a Hamiltonian system.

\subsection{Hamiltonian description} 

In what follows, we denote by
$\pi :  T^*\M\rightarrow \M$ the canonical projection.
Given a smooth vector field $X \in\Gamma ( T\M)$ on $M$, we denote by
\begin{align}\label{defcoordinateh}
\begin{array}{rccl}
 h_{X} :T^*\M\to\mathbb{R},\qquad 
 h_{X}(\xi)=\langle\xi, X( q ) \rangle,
\end{array}
\end{align}
the linear function on fibers associated with $X$, where $q=\pi(\lambda)$.
Given  an orthonormal frame $ X_1,X_2 $ of the distribution and  the Reeb field $X_{0}$, we consider the associated functions $h_{X_{i}}$ for $i=0,1,2$, and we define the \emph{sub-Riemannian Hamiltonian} $H : T^*\M \rightarrow \mathbb{R}$ as follows
\begin{align}\label{Hamiltonianfunction}
H=\frac{1}{2}( h_{X_1}^2 +h_{X_2}^2 ).
\end{align} 
One can show that $H$ actually does not depend on the choice of the frame $ X_1,X_2 $.

The cotangent bundle $T^{*}M$ is canonically endowed with its canonical symplectic structure $\sigma$, which in turns defines the Hamiltonian vector field $\vec H$ through the identity
$\sigma(\cdot, \vec H)=dH$.
In standard coordinates $(p,x)$ on $T^{*}M$ we have
$$
\vec H=\frac{\partial H}{\partial p}\frac{\partial }{\partial x}-\frac{\partial H}{\partial x}\frac{\partial }{\partial p}.
$$

We introduce a frame on the cotangent bundle $T^{*}M$ that is adapted to the choice of an orthonormal frame on $M$. 
Every vector field $X$ on $\M$ can be lifted to a vector field $\overline{X}$ on $ T^*\M$ by requiring
\begin{align}\label{defliftedfield}
\pi_{*} \overline{X}= X \quad  \text{ and } \quad \overline{X} h_{X_j} =0, \ \forall \,j.
\end{align}
We introduce then the frame of $T(T^*\M)$ defined by 
\begin{align*}\left( \overline{X}_0,\overline{X}_1,\overline{X}_2,\frac{\partial}{\partial h_{X_0}},\frac{\partial}{\partial h_{X_1}}, \frac{\partial}{\partial h_{X_2}} \right).
\end{align*}
Notice that $\frac{\partial}{\partial h_{X_i}}$ denotes the vertical vector field on $T^*\M$ satisfying
\begin{align}\label{defpartialhXi}
\pi_{*} \left( \frac{\partial}{\partial h_{X_i}}\right)=0,\quad  \text{ and } \quad\frac{\partial}{\partial h_{X_i}} h_{X_j}=\delta_{i,j}.
\end{align}
We stress that this frame does depend on a choice of orthonormal frame of the distribution. Similarly, we can lift a function
$f :\M\rightarrow \mathbb{R}$ to the function $\overline{f}=f\circ\pi:T^*\M\rightarrow \mathbb{R}$.
In particular, $\overline{c}_{i,j}^{k}$ means ${c}_{i,j}^{k}\circ \pi$.

The sub-Riemannian Hamiltonian vector field is expressed in the lifted frame as :
\begin{align}\label{expressionHamiltonian}
\overrightarrow{H}=\sum_{i=1}^2 h_{X_i} \overline{X}_i +\sum_{i=1}^{2}\sum_{j,k =0}^{2} \overline{c}_{i,j}^{k}h_{X_i} h_{X_k} \frac{\partial}{\partial h_{X_j}},
\end{align}
\begin{rem}
(i) The flow of the Hamiltonian vector field $\overrightarrow{H}$ is the geodesic flow, in the following sense: a horizontal curve
$\hc : \interv\rightarrow \M$
is a geodesic parametrized by constant speed if and only if there exists a lift
$
 \overline{\hc} : \interv\rightarrow T^*\M
$ such that 
$
 \pi\circ\overline{\hc} =\hc
$ and $
 \overline{\hc}'(t) = \overrightarrow{H}( \overline{\hc}(t) )
$, for every $t\in\interv$.
This result is classical, a proof can be found, for instance, in \cite[Chapter 4, Theorem 4.25]{ABB}.

(ii)  Geodesics parametrized by arc length in $\M$ are projections of integral lines of $\overrightarrow H$ contained in $H^{-1}(1/2)$. Indeed one can check by combining \eqref{Hamiltonianfunction} and \eqref{expressionHamiltonian} that
$H$ is constant along its Hamiltonian flow,
and $\| \pi_{*}  \overrightarrow{H}( \overline{\hc}(t) )  \|^2=2H( \overline{\hc}(t) )$.
%
%
%
\end{rem}

\subsection{Proof of Proposition~\ref{propcharacterizationgeod}}
To prove Proposition~\ref{propcharacterizationgeod} we have to show that a smooth horizontal curve $\hc :\interv\to\M$ is a projection of a solution of the Hamiltonian system defined by $H$ if and only if $\kappa_{\hc}=0$ along the curve.

Assume that $\overline{\hc}$ is a lift of $\hc$, and that $\overline{\hc}$ satisfies the Hamiltonian equation  
\begin{equation}\label{eq:HHHH}
\overline{\hc}'(t) = \overrightarrow{H}( \overline{\hc}(t) ).
\end{equation}
Using  expression \eqref{expressionHamiltonian} and projecting along the orthonormal frame $X_{1}=\T$, $X_{2}= J\T $,
we have
\begin{align} \label{eq:234}
 \hc'(t)=h_{\T }( \overline{\hc} (t) )\T ( \hc (t) ) +h_{J\T }( \overline{\hc} (t) )J\T ( \hc (t) ).
\end{align}
By the definition of $\T $, one has 
\begin{align}\label{eq:TJT}
h_{\T }( \overline{\hc} (t) )=1\text{ and }h_{J\T }( \overline{\hc} (t) )=0.
\end{align}
and combining this with \eqref{expressionHamiltonian} and \eqref{eq:HHHH}, we find (recall $c_{\T,J\T}^{X_{0}}=1$)
\begin{align}
 \overline{\hc} '(t)=&\overline{\T }( \overline{\hc}(t) ) -c_{X_0,\T }^{\T }( \hc(t) )\frac{\partial}{\partial h_{X_0}}\label{equalityiii}
+( c_{\T ,J\T }^{\T }( \hc(t) ) +h_{X_0}( \overline{\hc}(t) ) )\frac{\partial}{\partial h_{J \T }}.
\end{align}
From the latter we deduce
\begin{align}\label{systemkzetazero}
 \left\lbrace\begin{array}{rl}
  \frac{\emph{d}}{\emph{d}t}h_{X_0}( \overline{\hc} (t) )&=-c_{X_0,\T }^{\T }( \hc(t) )=\eta( \hc'(t) )\\[0.2cm]
 0=\frac{\emph{d}}{\emph{d}t}h_{J\T}( \overline{\hc} (t) )&=c_{\T ,J\T }^{\T }( \hc(t) ) +h_{X_0}( \overline{\hc}(t) )=-h_{\hc}(t) +h_{X_0}( \overline{\hc}(t) ) .
 \end{array}\right.
\end{align}
which implies
\[
\kappa_{\hc}(t)=\frac{\emph{d}}{\emph{d}t}h_{\hc}(t)-\eta( \hc'(t) )=0.
\]
The converse is proved in a similar way by building a lift satisfying \eqref{eq:TJT}-\eqref{systemkzetazero}.
\begin{rem}\label{ProphX0geod}
 It follows from the proof that if $\hc$ is a geodesic then the characteristic deviation coincides with the evaluation of $h_{X_{0}}$ along its lift, namely
  \begin{align} \label{eq:negativo}
h_{\hc}(t)= h_{X_0}( \overline{\hc}(t) ).
 \end{align}
Recall that the coordinate function $h_{X_0}$ defined on $T^*\M$ does not depend on the choice of an orthonormal frame $( X_1,X_2 )$ of the distribution and can be regarded as a vertical component of the covector.
\end{rem}

\section{About distance and cut locus}\label{s:distcut}
In this section we recall some results on the cut locus of sub-Riemannian distance, which are needed later to prove regularity of the distance function along a smooth horizontal curve. We also refer to \cite[Chapter 11]{ABB} for more details.

\medskip

In what follows we fix $p$ a privileged point in $\M$, that will be the origin of our smooth horizontal curve.
 The sub-Riemannian distance from $p$ is denoted by
\begin{equation}
   \delta_{p} : \M\to \mathbb{R}^+ ,\qquad 
   \delta_{p}(q)=d_{\SR}( p,q ).
\end{equation}
We work on a compact neighborhood of $p$, in such a way that we can assume without loss of generality that the sub-Riemannian structure is complete. 

We denote by $\Sigma_p$ the set of smooth points of $\delta_{p}$.
The following result is proved in \cite{Agrasmoothness} (see also \cite[Chapter 11]{ABB}). 
\begin{theo} Let $p\in M$ and denote by $\delta_{p}$ the sub-Riemannian distance from $p$.
\begin{itemize}
\item[(a)] the set $\Sigma_p$ of smooth points of $\delta_{p}$ is open and dense in $\M$, 
\item[(b)] we have $q\in \Sigma_{p}$ if and only if there exists a unique arc length
geodesic joining $p$ and $q$ that is not abnormal and not conjugate,
\item[(c)] if $q\in \Sigma_{p}$ and $\overline{\geo}$ is the lift on $T^{*}M$ of the unique arc length geodesic $\geo$ joining $p$ to $q$ (in time $\delta_{p}(q))$ then
\begin{equation}\label{eq:ddelta}
\emph{d}_{q}\delta_p=\overline{\geo}( \delta_{p}(q) ).
\end{equation}

\end{itemize}
\end{theo}
Notice that since we consider contact sub-Riemannian structures, there are no nontrivial abnormal length-minimizers, and one can prove that $M\setminus \Sigma_p$ has measure zero, cf.\ \cite[Chapter 11]{ABB}.

 For any $q\in\Sigma_p$, we denote by $\geo_q$ the unique arc length geodesic reaching $q$ in time $\delta_{p}(q)$.
 In what follows we drop $p$ from the notation of the distance $\delta_{p}$.
\begin{de} \label{defgammaQFieldGammaradialcoord}
For every $q$ in $\Sigma_p$ we define the \emph{radial vector field} $\rvf$ by
 \begin{align*}
\rvf:\Sigma_{p}\to TM,\qquad  \rvf( q ) =\grad_{\SR}\delta (q).
 \end{align*}
and the \emph{radial deviation} $\rd$ by
\[ \rd:\Sigma_{p}\to \mathbb{R}, \qquad \rd(q)=h_{\geo_{q}} ( \delta( q ) ).
\]
 \end{de}
 
 We stress that the radial deviation  $\rd (q)$ represents the characteristic deviation of the geodesic with unit speed joining $p$ and $q$.

\smallskip
The following lemma contains some information on  $\rvf$ and $\rd$ at smooth points.
\begin{lem} \label{propgraddelta} \label{ProphGammae0delta} On $\Sigma_p$, the vector field $\rvf$ is smooth and  satisfies:
\begin{itemize}
\item[(i)] $\rvf \delta=1$,  $J\rvf \delta=0$ 
\end{itemize}
Moreover the function $\rd$ is smooth on $\Sigma_p$ and satisfies the identities:
\begin{itemize}
\item[(ii)] $\rd=X_{0} \delta$,
\item[(iii)] $\rvf \rd =\eta ( \rvf )$,
\item[(iv)] $J\rvf \rd =c_{J\rvf, X_0}^{\rvf}= 2\iota( \rvf  ) -c_{\rvf,X_0}^{J\rvf}$.
\end{itemize}
\end{lem}
\begin{proof} 
 (i).\ We have $\rvf \delta =g(\rvf,\grad_{\SR}\delta)=\|\Gamma\|^{2}=1$. Similarly $J\rvf \delta =g(J\rvf,\rvf)=0$ on $\Sigma_p$. The smoothness of $\rvf$ is a consequence of the smoothness of $\delta$ on $\Sigma_p$.

(ii).\ Over $\Sigma_p$, where the field $\rvf$ and the function $\delta$ are smooth, we can write
\begin{align*}
0&= \rvf \left( J\rvf \delta \right) -J\rvf \left( \rvf \delta \right) \text{ (thanks to (i))} \\
&=\left[ \rvf, J\rvf \right]\delta \\
&= c_{\rvf,J\rvf}^{\rvf} \rvf \delta + c_{\rvf,J\rvf}^{J\rvf} J\rvf \delta + c_{\rvf,J\rvf}^{X_0} X_0 \delta\\
&=-\rd + X_0 \delta \text{ (by (i)).}
\end{align*}
where in the last line we used that $\rd=-c_{\rvf,J\rvf}^{\rvf}$, claim (i), and  $c_{\rvf,J\rvf}^{X_0}=1$. 

(iii).\ Let us compute 
\begin{align*}
\rvf \rd &= \rvf \left( X_0 \delta \right)\text{ (thanks to (ii))}\\
&=  \rvf \left( X_0 \delta \right) - X_0 \left( \rvf \delta \right) \\
&=\left[ \rvf , X_0 \right]\delta =c_{\rvf , X_0}^{\rvf} \rvf \delta +c_{\rvf , X_0}^{J\rvf} J\rvf \delta + c_{\rvf , X_0}^{X_0} X_0 \delta\\
&=c_{\rvf , X_0}^{\rvf}=\eta(\rvf), 
\end{align*}
where we used claim (i), $c_{\rvf , X_0}^{X_0} X_0 =0$ and \eqref{eq:etaiota1} for $\eta$. 

(iv).\ Similarly as in (iii), let us write
\begin{align*}
J\rvf \rd &= J\rvf \left( X_0 \delta \right) =  J\rvf \left( X_0 \delta \right) - X_0 \left( J\rvf \delta \right) \\
&=\left[ J\rvf , X_0 \right]\delta =c_{J\rvf , X_0}^{\rvf} \rvf \delta +c_{J\rvf , X_0}^{J\rvf} J\rvf \delta + c_{J\rvf , X_0}^{X_0} X_0 \delta=c_{J\rvf , X_0}^{\rvf} 
\end{align*}
where we used claim (i) and $c_{J\rvf , X_0}^{X_0} X_0 =0$. This corresponds to the first part of the second identity. The second part of the identity is a  consequence of  \eqref{eq:etaiota1}.
\end{proof}

\subsection{On the cut locus} Recall that given 
 a geodesic  $\geo : [0,T]\rightarrow\M$ parametrized by arc length we define the \emph{cut time}  as follows
\begin{align*}
t_{\mathrm{cut}}( \geo )=\sup \{ t>0 \mid \ell_{\SR}(\geo_{|_{[0,t]}})=d_{\SR}(\geo( 0 ),\geo( t ))\}.
\end{align*}
The cut point along $\geo$ is the point $\geo(t_{\mathrm{cut}}( \geo ))$, and the cut locus from a point $p$ is the set $\mathrm{cut}(p)$ of all cut points of arc length geodesic starting from $p$.

For a 3D contact sub-Riemannian manifold one has $M\setminus \mathrm{cut}(p)=\Sigma_{p}$.
 The cut locus is strictly related to the singularities of the exponential map.
 \begin{de} We define the exponential map $\exp:S_{p}M\times \mathbb{R} \times \mathbb{R}^{+}\to M$ as follows: given $(v,h,t)\in S_{p}M\times \mathbb{R} \times \mathbb{R}$ we define $\exp(v,h,t)$ as the value at time $t>0$ of the geodesic with initial vector $v$ and constant characteristic deviation $h$.
 \end{de}
Recall that $S_{p}M$ is the set of unit horizontal vectors at $p$. Usually the exponential map is defined on the set of covectors $T^{*}_{p}M$, or on $S^{*}_{p}M\times \mathbb{R}$ where $S^{*}_{p}M=T^{*}_{p}M \cap H^{-1}(1/2)$, with $H$  the sub-Riemannian Hamiltonian. Here we prefer to parametrize by a unit horizontal vector and its characteristic deviation, which is constant.

 It is well known that the exponential map is smooth. We refer to \cite[Chapter~8]{ABB} for a comprehensive discussion on these results.
 \section{Expansion of the distance}

In this section we fix $\hc :I\rightarrow M$ a smooth horizontal curve parametrized by arc length such that $\hc(0)=p$. We define $\theta :I\setminus\left\{ 0 \right\}\rightarrow S^1$ as the function such that :
\begin{equation}\label{deftheta}
\hc'(t)=\cos( \theta(t) )\rvf( {\hc}(t) )  + \sin( \theta(t) ) J\rvf ( {\hc}(t) ).
\end{equation}
We suppose that $\hc :I\rightarrow M$ satisfies the following assumption
\begin{itemize}
\item[(H)]  $\hc(t)\notin \mathrm{cut}(\hc(0))$ for $t\neq 0$ small. 
\end{itemize} 
As we will prove in Section~\ref{s:crd}, these assumptions are always  satisfied for a 
sufficiently short arc of a smooth horizontal curve parametrized by arc length. 

\medskip
We stress that the vector fields $\Gamma$ and $J\Gamma$ are singular at $\hc(0)$, hence the regularity of the function $t\mapsto \theta(t)$ at $t=0$ is not guaranteed.
Assuming the $C^{2}$ regularity of the function, it  is not difficult to obtain a Taylor expansion of the distance along the curve $\hc$ in terms of the derivatives of $\theta$. More precisely we have the following result.
\begin{prop}\label{PropTaylorTheta} Let $\hc:I\rightarrow M$ be a smooth horizontal curve parametrized by arc length satisfying (H). 
Assume that the function $\theta$ defined by \eqref{deftheta} can be extended to a $\mathcal{C}^2$ function such that
$\theta(0)={\theta}'(0)=0$.
Then
\begin{align}\label{eq:mainth}
d^{2}_{\SR}( \hc(t),\hc(0) )=t^{2}-\frac{{\theta}''(0)^2}{20}t^6 +o( t^6 ).
\end{align}
\end{prop}
\begin{proof} Fix $p=\geo(0)$. We have $d_{\SR}( \hc(t),\hc(0) )=\delta \left( \hc(t) \right)$ and 
\begin{align*}
\delta \left( \hc(t) \right)
&=\int_0^t g( \grad_{\SR} \delta( \hc(s) ),\hc'(s)  )\text{d}s \\
&=\int_0^t \cos\left( \theta(s) \right)\text{d}s =\int_0^t \cos\left( \frac{{\theta}''(0)}{2} s^2 +o( s^2 ) \right)\text{d}s \\
&=t-\frac{{\theta}''(0)^2}{40}t^5 +o( t^5 ). \qedhere
\end{align*}
\end{proof}

The goal of the following sections is to show that the assumptions of Proposition~\ref{PropTaylorTheta} are satisfied for any smooth horizontal curve $\hc$ parametrized by arc length. Moreover, we recover the geometric meaning of the nontrivial coefficient appearing in \eqref{eq:mainth}.

\medskip
We first relate the characteristic deviation and the geodesic curvature with $\theta$. 

\begin{prop} \label{proplinkhktheta} Let $\hc:I\rightarrow M$ be a smooth horizontal curve parametrized by arc length. For any $t\neq 0$ we have
\begin{align}\label{equalinkhtheta}
h_{\hc}(t)&={\theta}'_{\hc}(t)+\cos(\theta(t))\rd ( \hc(t) ) -\sin( \theta(t) )c_{\rvf,J\rvf}^{J\rvf} ( \hc(t) )
\end{align}
Moreover
\begin{align}\label{equalinkktheta}
k_{\hc}(t)&=-\eta ( \hc'(t) ) +\theta''(t) - {\theta}'(t) ( \sin( \theta(t) )\rd( \hc(t) )  +\cos( \theta(t) )c_{\rvf,J\rvf}^{J\rvf}( \hc(t) )   )\\
&\nonumber\quad +\cos^2( \theta(t) ) (\cos ( 2\theta(t) ) \eta ( \hc'(t) )  + \sin ( 2\theta(t) ) \iota ( \hc'(t) )
) \\
&\nonumber \quad +\sin( 2\theta(t) )  \left(  -\sin ( 2\theta(t) ) \eta ( \hc'(t) )
+\cos ( 2\theta(t) ) \iota ( \hc'(t) ) -\frac{1}{2}c_{\rvf,X_0}^{J\rvf} \right)\\
&\nonumber \quad -\frac{\sin( 2\theta(t) )}{2} ( \rvf c_{\rvf,J\rvf}^{J\rvf} )( \hc(t) )
-\sin^2( \theta(t) ) ( J\rvf c_{\rvf,J\rvf}^{J\rvf} )( \hc(t) ).
\end{align}
\end{prop}


\begin{proof}
By Lemma~\ref{lemma10}, using the frame $\{X_{1},X_{2}\}=\{\rvf,J\rvf\}$ we have  for any $t\neq 0$,
 \begin{align*}
  h_{\hc}(t)&= \theta' ( t ) - \cos( \theta(t) ) c_{\rvf,J\rvf}^{\rvf} ( \hc(t) ) - \sin( \theta(t) ) c_{\rvf,J\rvf}^{J\rvf}( \hc(t) ) \\
  &=\theta' ( t ) +\cos( \theta(t) )\rd(\hc (t)) - \sin( \theta(t) ) c_{\rvf,J\rvf}^{J\rvf}( \hc(t) ). 
 \end{align*}
 This proves \eqref{equalinkhtheta}.
 To obtain \eqref{equalinkktheta}, we start from Definition \ref{defgeocurv}:
 \begin{align}\label{kzetathetabegincomputation}
 k_{\hc}&= -\eta ( \hc'(t) ) +h_{\hc}'(t) \nonumber \\
  &=-\eta ( \hc'(t) ) +\theta''(t) - {\theta}'_{\hc}(t) ( \sin( \theta(t) )\rd( \hc(t) )  +\cos( \theta(t) )c_{\rvf,J\rvf}^{J\rvf}( \hc(t) )   ) \\
  & \quad+\cos( \theta(t) )(\text{d}\rd)(\hc' (t)) - \sin( \theta(t) ) ( \text{d}c_{\rvf,J\rvf}^{J\rvf})( \hc'(t) ). \nonumber
 \end{align}
We now focus on the two terms $(\text{d} c_{\rvf,J\rvf}^{J\rvf} )(\hc' (t))$ and $(\text{d}\rd)(\hc' (t))$.
We replace the vector $\hc'(t)$ by its expression in the frame $( \rvf,J\rvf )$ in terms of $\theta$. We obtain
\begin{align}
 (\text{d} c_{\rvf,J\rvf}^{J\rvf} )(\hc' (t))&= \cos( \theta(t) ) ( \rvf c_{\rvf,J\rvf}^{J\rvf} )( \hc(t) )
+\sin( \theta(t) ) ( J\rvf c_{\rvf,J\rvf}^{J\rvf} )( \hc(t) ), \label{kzetathetaI}
\end{align}
hence
\begin{align*}
 - \sin( \theta(t) ) (\text{d} c_{\rvf,J\rvf}^{J\rvf} )(\hc' (t))&= -\frac{\sin( 2\theta(t) )}{2} ( \rvf c_{\rvf,J\rvf}^{J\rvf} )( \hc(t) )
-\sin^{2}( \theta(t) ) ( J\rvf c_{\rvf,J\rvf}^{J\rvf} )( \hc(t) ).
\end{align*}
Moreover
\begin{align}
 \nonumber(\text{d}\rd)(\hc' (t)) &= \cos( \theta(t) )(\rvf \rd)( \hc(t) ) +\sin( \theta(t) )(J\rvf \rd)( \hc(t) )\\
 &\nonumber=\cos( \theta(t) )\eta( \rvf( \hc(t) ) )  +\sin( \theta(t) ) ( 2\iota ( \rvf ( \hc(t) ) ) -c_{\rvf,X_0}^{J\rvf} ( \hc(t) )  ),
\end{align}
where we used Lemma~\ref{propgraddelta}. Using \eqref{deftheta} and Lemma~\ref{etaiota}, we deduce that 
\begin{align*}
 \cos( \theta(t) ) (\text{d}\rd)(\hc' (t))&=
 \cos^2( \theta(t) ) ( \cos ( 2\theta(t) ) \eta ( \hc'(t) )+ \sin ( 2\theta(t) ) \iota ( \hc'(t) ) ) \label{kzetathetaII}\\
 & +\sin( 2\theta(t) )  (  -\sin ( 2\theta(t) ) \eta ( \hc'(t) )
+\cos ( 2\theta(t) ) \iota ( \hc'(t) ) -\frac{1}{2}c_{\rvf,X_0}^{J\rvf} ).
\end{align*}
and the proof is completed by combining the last identity with \eqref{kzetathetabegincomputation} and  \eqref{kzetathetaI}.
\end{proof}

%
%

\subsection{Continuity of the radial deviation coordinate} \label{s:crd}
Let us start by rewriting identity \eqref{equalinkhtheta} as follows
\begin{align}\
{\theta}'_{\hc}(t)&=h_{\hc}(t)-\cos(\theta(t))\rd ( \hc(t) ) +\sin( \theta(t) )c_{\rvf,J\rvf}^{J\rvf} ( \hc(t) ).
\end{align}

We  prove here the continuity of the radial deviation coordinate along a smooth horizontal curve parametrized by arc length leaving from $p$. More precisely 
\begin{equation} \label{eq:hh0}
\lim_{t\to 0}\left(h_{\hc}(t)-\rd ( \hc(t) )\right)=0.
\end{equation}
We need the following lemma, whose proof follows from the standard Taylor-Lagrange formula. 
\begin{lem} \label{LemmaRegularity}
Let 
$
  F : \mathbb{R}^d \times \mathbb{R} \to \mathbb{R}
$
be a $C^\infty$ function such that for any $x$ in $\mathbb{R}^d $,
\begin{align*}
 F ( x,0 ) =\frac{\partial F}{\partial y}( x,0)   = \cdots =\frac{\partial ^{n-1} F}{\partial y ^{n-1}}( x,0)  =0.
\end{align*}
Then the function $\mathbb{R}^d \times \mathbb{R} \to \mathbb{R}$ defined by $(x,y)\mapsto \frac{F ( x,y ) }{y^n}$ is of class $C^\infty$.
\end{lem}

To prove \eqref{eq:hh0} we use normal coordinate. We need the following lemma (cf.\ also with Proposition~\ref{propcharacteristicdevGauthier}).
\begin{lem} \label{LemmahGauthier}In normal coordinates $\left(x,y,z\right)$ around $p$ the map :
\begin{align*}
\begin{array}{rccc}
\Psi:S_p \M \times \mathbb{R} \times \mathbb{R}\to  \mathbb{R}\qquad
\Psi\left( v,h,t \right) = \frac{12 z\left( \exp\left( v,h,t  \right) \right)}{\left( x^2\left( \exp\left( v,h,t  \right) \right) +y^2\left( \exp\left( v,h,t  \right) \right) \right)^{3/2}}
\end{array}
\end{align*}
is smooth and $\Psi( v,h,0 )=h$.
\end{lem} 
\begin{proof}
The exponential map is smooth. Thanks to Proposition \ref{propcharacteristicdevGauthier} combined with Lemma \ref{LemmaRegularity} we have that 
 \begin{align*}
\begin{array}{rccc}
\Psi_{1}:S_p \M \times \mathbb{R} \times \mathbb{R}\to  \mathbb{R}\qquad
\Psi_{1}\left( v,h,t \right) = \frac{12z\left( \exp\left( v,h,t  \right) \right)}{t^3}
\end{array}
\end{align*}
is smooth and  $\Psi_{1}\left( v,h,0 \right)=h$.
Moreover, for any $\left( v,h \right)\in S_p \M \times \mathbb{R}$, we have that $t\mapsto\exp\left( v,h,t  \right)$ is parametrized by arc length, equal to $p$ at
$t=0$. By the expression of the orthonormal frame of the distribution in normal coordinates around $p$ given in Proposition~\ref{defGauthier}, we deduce that for $ (v,h )$ fixed
\begin{align*}
 \lim_{t\to 0}\frac{1}{t}\left(x^2\left( \exp\left( v,h,t  \right) \right) +y^2\left( \exp\left( v,h,t  \right) \right) \right)^{1/2}=1.
\end{align*}
Therefore, by using once more Lemma \ref{LemmaRegularity}, we obtain that
\begin{align*}
\begin{array}{rccc}
\Psi_{2}:S_p \M \times \mathbb{R} \times \mathbb{R}\to  \mathbb{R}\qquad
\Psi_{2}\left( v,h,t \right) =\frac{ (x^2\left( \exp\left( v,h,t  \right) \right) +y^2\left( \exp\left( v,h,t  \right) \right))^{1/2} }{t}\end{array}
\end{align*}
is smooth and $\Psi_{2}\left( v,h,0 \right)=1$, which concludes the proof.
\end{proof}

\subsection{Horizontal curves do not intersect the cut locus for small times}
 We prove now that the characteristic deviation of geodesics joining $\hc(0)$ with $\hc(t)$ converge to the characteristic deviation of $\hc$ when $t\to 0$. 
 
 As a byproduct we also prove that every horizontal curve does not intersect the cut locus for small times, namely satisfies assumption (H).
 \begin{prop}\label{Proplimith0}
  Let $\hc :I\to M$ be a smooth horizontal curve parametrized by arc length. Then we have
  \begin{itemize}
  \item[(i)] $\hc(t)\notin \mathrm{cut}(\hc(0))$   for $t\neq 0$ small enough, 
\item[(ii)] 
$\lim_{t\to 0}   \rd\left( \hc(t) \right) = h_{\hc}(0)$.
  \end{itemize}
 \end{prop}
\begin{proof}
 For $t\in I\setminus \{0\}$, let us denote $\geo_{t}$ the length-minimizing geodesic joining $p$ and $\hc(t)$ in time $\delta(\hc(t))$. For $t$ small enough these minimizers are all contained in a closed ball $B$. Writing
 \begin{equation}\label{eq:new}
 |\rd(\hc(t) )-h_{\hc}(0)|\leq |\rd(\hc(t) )-h_{\geo_{t}}(0)|+|h_{\geo_{t}}(0)-h_{\hc}(0)|
 \end{equation}
Since by definition of $\rd$ we have that $\rd(\hc(t) )=h_{\geo_{t}}(\delta(\hc(t)))$, the first term in \eqref{eq:new} tends to zero when $t\to 0$ as
$$|h_{\geo_{t}}(\delta(\hc(t)))-h_{\geo_{t}}(0)|\leq \int_{0}^{\delta(\hc(t))}\left|\frac{d}{ds}h_{\geo_{t}}(s)\right|ds \leq \sup_{B} |\eta|\delta(\hc(t))\to 0$$
where we used that along a geodesic $\frac{d}{ds} h_{\geo}(s)=\eta( \geo' (s))$ thanks to the geodesic equation and $\delta(\hc(t))\to 0$ for $t\to 0$ by continuity of $\delta$. It is sufficient then to prove that
   \begin{align}\label{eq:limith0no} 
\lim_{t\to 0}  h_{\geo_{t}}(0) = h_{\hc}(0)=: h_{0}.
  \end{align}
 Let us now assume that \eqref{eq:limith0no} is not true.
 Without loss of generality, we can assume that there exist $\varepsilon >0$ and a sequence $t_{n}\to 0$ such that 
 \begin{align}\label{eq:3eps}
  h_{\geo_{n}}(0) \geqslant  h_{0}+  3\varepsilon.
 \end{align}
where we denoted for simplicity by $\geo_{n}$ the curve $\geo_{t_n}$.

Let us now fix normal coordinates near $p$ and set $\Phi(x,y,z)=12z/( x^2 +y^2 )^{\frac{3}{2}}$. Define the super- and sub-level set
\begin{align*}
\Phi^{+}_{\lambda_{1}}=\{ \Phi \geq \lambda_{1}  \},\qquad 
\Phi^{-}_{\lambda_{2}}=\{ \Phi \leq \lambda_{2}  \},\qquad \Phi_{\lambda_{1},\lambda_{2}}=\{\lambda_{1}\leq  \Phi \leq \lambda_{2}  \}.
\end{align*}
By definition of characteristic deviation and Proposition~\ref{propcharacteristicdevGauthier}, one has $\hc(t_n)\in \Phi^{-}_{ h_{0} +\varepsilon}$ for $n$ large enough.
On the other hand, thanks to \eqref{eq:3eps}, for  $s$ small enough one has $\geo_{n}(s)\in \Phi^{+}_{ h_{0} +3\varepsilon}$.

Since $\geo_{n}$ reaches the point $\hc( t_n )$ (at time $\delta( \hc( t_n ) )$), this means that for $n$ large enough, $\geo_{n}$ passes from the set $\Phi^{+}_{ h_{0} +3\varepsilon}$ to the set $\Phi^{-}_{ h_{0} +\varepsilon}$. We are now going to show that this gives a contradiction. We need the following:  

\smallskip
\emph{Claim}: there exists $\tau_{0}>0$ such that for every $v\in S_{p}M$ the curve $t\mapsto\exp(v,h_{0}+2\eps,t)$ is optimal and belongs to $\Phi_{h_{0}+\eps,h_{0}+3\eps}$ on $0<t\leq \tau_{0}$.

\smallskip
Assume that the Claim is proved and let us conclude the proof. On one hand $\geo_{t}$ is the length-minimizing geodesic joining $p$ and $\hc(t)$ in time $\delta(\hc(t))$, hence $\geo_{t}(s)\notin \mathrm{cut}(p)$ for $s<\delta(\hc(t))$.
 On the other hand, for $t$ sufficiently small, $\geo_{t}$ passes from the set $\Phi^{+}_{ h_{0} +3\varepsilon}$ to the set $\Phi^{-}_{ h_{0} +\varepsilon}$, hence by continuity must cross the set
 \begin{align} \label{eq:s2e}
  \mathcal{S}_{2\eps}:=\exp\left( S_p \M \times \left\lbrace h_0 +2\varepsilon \right\rbrace \times [0,\tau_{0}] \right).
 \end{align}
This implies $\delta( \hc(t_n) ) \geq \tau_{0}$ for every $n$ large enough, which is clearly a contradiction since $\delta( \hc(t_n) )\to 0$. Notice that $\mathcal{S}_{2\eps}$ is contained in $\Phi_{h_{0}+\eps,h_{0}+3\eps}$ for $\tau_{0}$ small enough.

The proof is concluded by the existence of $\tau_{0}>0$ in the Claim. This is a consequence of the continuity of the cut time with respect to initial conditions in absence of abnormal minimizers (cf.\ \cite[Proposition~8.76]{ABB}), and the compactness of the set of initial data in \eqref{eq:s2e}.
\end{proof}

\subsection{Asymptotics of Lie brackets}
Now we give a statement about the asymptotics of the coefficients of the Lie brackets of the elements of the frame $\left( \rvf,J\rvf \right)$ along a horizontal curve parametrized by arc length leaving from $p$. 
\begin{prop}\label{asymptoticsLieprojected}
 Let  $\hc:I\to M$ be a smooth horizontal curve parametrized by arc length such that $\hc(0)=p$. Then for $t\to 0$
\begin{itemize}
\item[(a)] $\delta\left( \hc(t) \right)c_{\rvf,J\rvf}^{J\rvf}\left( \hc(t) \right)\longrightarrow -4$,
\item[(b)] $\delta^2\left( \hc(t) \right)c_{\rvf,X_0}^{J\rvf}\left( \hc(t) \right)\longrightarrow -6$,
\item[(c)] $\delta^2\left( \hc(t) \right)\rvf c_{\rvf,J\rvf}^{J\rvf}\left( \hc(t) \right)\longrightarrow 4$, 
\item[(d)] $\delta^2\left( \hc(t) \right)J\rvf c_{\rvf,J\rvf}^{J\rvf}\left( \hc(t) \right) =O(1)$.
\end{itemize}
\end{prop}
The proof of Proposition~\ref{asymptoticsLieprojected} relies on properties of sub-Riemannian Jacobi fields. We give a self-contained proof in  Appendix~\ref{s:appjac}.
\section{Regularity along a smooth horizontal curve: proof of Theorem~\ref{t:main}}
We now go back to the regularity properties of
the function $\theta$, which satisfies
\begin{equation}\label{deftheta2}
\hc'(t)=\cos( \theta(t) )\rvf( {\hc}(t) )  + \sin( \theta(t) ) J\rvf ( {\hc}(t) ).
\end{equation}
We first prove a technical lemma.
\begin{lem}\label{thetanotnegative} For every $\eps>0$ there exists $\bar t\in (0,\eps)$ such that $\cos\left(\theta(\bar t)\right)\geq 0$.
\end{lem}
\begin{proof}
Assume by contradiction that there exists $\eps_0>0$ such that $\cos(\theta(t))<0$ for every $t$ in $(0,\eps_0)$.
 By combining Lemma \ref{propgraddelta} and Definition \ref{deftheta}, for $t\in (0,\eps_{0})$
\begin{align*}
\frac{\text{d}}{\text{d}t} \delta( \hc(t) ) = \cos( \theta(t) )<0.
\end{align*}
Since $\delta( \hc(0) )=0$ (recall $\hc(0)=p$) and $\delta$ is continuous, for $t$ positive and small enough one has
$\delta( \hc(t) )<0$,
which is a contradiction.
\end{proof}

\subsection{First order} The goal of this section is to prove that the function $\theta$ can be extended continuously at zero in such a way that $\theta(0)=0$, i.e., we have to show the existence of the limit
$$\lim_{t\to 0}\theta(t)=0.$$
Geometrically, this is saying that the tangent vector at zero of the geodesics joining $\hc(t)$ to $\hc(0)$ converges to $ \hc'(0)$ when $t\to 0$.

\smallskip
We define the following quantity, for every $t>0$:
$$M(t):=\sup_{s\in]0,t]} \left\lvert h_{\hc}(s) -\cos( \theta(s) ) \rd( \hc(s) ) \right\rvert. $$
Notice that $M(t)$ is bounded for $t\to 0$ thanks to Proposition~\ref{Proplimith0}.
\begin{lem}\label{lemmadiffequation}
Let $\hc:I\rightarrow M$ be a smooth horizontal curve parametrized by arc length such that $\hc(0)=p$. Then
for any $t\neq 0 $ small enough 
\begin{align*}
 \left\lvert \sin\left( \theta(t) \right) \right\rvert \leqslant t \frac{M(t)}{3}.
\end{align*}
\end{lem}
\begin{proof}
Let us prove the lemma by contradiction. Let $t_{0}>0$ (which can be chosen arbitrarily small) be such that 
\begin{align*}
 \left\lvert \sin\left( \theta(t_0) \right) \right\rvert > t_0 \frac{M(t_{0})}{3}=:\overline M.
\end{align*}
We decompose the circle $S^{1}$ into three disjoint sets (depending on $t_{0}$).
 \begin{align*}
 \mathcal{Z}_1&=\{\theta\in S^{1}: |\sin \theta|>\overline M\},\\
  \mathcal{Z}_2&=  \mathcal{Z}_1^{c}\cap \{\theta\in S^{1}: \cos \theta>0\},\\
  \mathcal{Z}_3&= \mathcal{Z}_1^{c }\cap \{\theta\in S^{1}: \cos \theta<0\}.
 \end{align*}
where $A^{c}$ denotes the complementary of a subset $A$ of the circle.
By Proposition \ref{asymptoticsLieprojected}, we can assume that for $t\leq t_{0}$ we have:
\begin{align}\label{deltacGammaJGammaJGammanearlimit}
 {\delta}\left( \hc(t) \right) c_{\rvf,J\rvf}^{J\rvf}\left( \hc(t) \right) < -3.
\end{align}
Assume for a moment that $\theta(\hc(t))\in \mathcal Z_{1}$ for every $0\leq t\leq t_{0}$. Then
\begin{align} \label{eq:equequ}
\left\lvert \sin\left( \theta(t) \right)c_{\rvf,J\rvf}^{J\rvf} \left( \hc(t) \right) \right\rvert&=
\left\lvert \frac{\sin\left( \theta(t) \right)}{\delta\left( \hc(t) \right)}\delta\left( \hc(t) \right)c_{\rvf,J\rvf}^{J\rvf} \left( \hc(t) \right) \right\rvert \nonumber\\
&> \left\lvert 3\frac{\sin\left( \theta(t) \right)}{\delta\left( \hc(t) \right)} \right\rvert \text{ (by \eqref{deltacGammaJGammaJGammanearlimit}),}\\
&\geqslant \frac{t_0}{\delta\left( \hc(t) \right)} M(t_{0}) \geqslant M(t_{0}),\nonumber
\end{align}
where we used that $0\leqslant \delta\left( \hc(t) \right)\leq t\leq  t_0$ (thanks to the fact that $\hc$ is arc length parametrized).
Thanks to Proposition \ref{proplinkhktheta}, we have
\begin{align}\label{Thediffequation}
{\theta}'_{\hc}(t)&=h_{\hc}(t)-\cos\left(\theta(t)\right)\rd \left( \hc(t) \right) +\sin\left( \theta(t) \right)c_{\rvf,J\rvf}^{J\rvf} \left( \hc(t) \right).
\end{align}
This implies that, if $\theta(\hc(t))\in \mathcal Z_{1}$ for every $0\leq t\leq t_{0}$, then
 ${\theta}'_{\hc}(t)\neq 0$ and has the same sign as the quantity $\sin\left( \theta(t) \right)c_{\rvf,J\rvf}^{J\rvf} \left( \hc(t) \right)$. By \eqref{deltacGammaJGammaJGammanearlimit}, this means that ${\theta}'_{\hc}(t)$ and $\sin\left( \theta(t) \right)$ have opposite signs,
 which means that
 \begin{align*}
 \frac{\text{d}}{\text{d}t}\cos\left( \theta(t) \right) = -\sin\left( \theta(t) \right){\theta}'_{\hc}(t)> 0.
 \end{align*}
%
%

Let us now go back to the proof. Our assumption says that $\theta(t_0)\in \mathcal Z_{1}$.
By the previous considerations on the orientation of the time-dependent vector field associated to the differential equation \eqref{Thediffequation} on the set $\mathcal Z_{1}$, there are two cases:
 \begin{itemize}
 \item[(i)] there exists $0\leq \bar t \leq t_{0}$  such that $\theta\left( \bar t \right)\in \mathcal{Z}_3$,
  \item[(ii)] for all $0\leq  t \leq t_{0}$, we have $\theta(t)\in \mathcal{Z}_1$.
 \end{itemize}
 
In case (i) we have that $\theta\left( \bar t \right)\in \mathcal{Z}_3$ for every $0\leq t\leq \bar t $. 
By definition of $\mathcal{Z}_3$, this implies that for all $0\leq t\leq \bar t $ 
one has $\cos\left( \theta(t) \right)<0$,
which is impossible, thanks  to Lemma \ref{thetanotnegative}.

In case (ii) we have that $t\mapsto \cos(\theta(t))$ is increasing on $(0,t_0]$, hence $\cos\left( \theta(t) \right)\to \alpha$ for $t\to 0$. Notice that by construction $\alpha\neq \pm1$, hence there exists $\sin \left( \theta(t) \right)\to \beta$, with $\beta\neq 0$.

Let us then rewrite \eqref{Thediffequation} as
\begin{align*}
\theta(t)-\theta(t_0)=\int_{t}^{t_0} h_{\hc}(s)-\cos\left(\theta(s)\right)\rd \left( \hc(s) \right) \text{d}s+\int_{t}^{t_0}\sin\left( \theta(s) \right)c_{\rvf,J\rvf}^{J\rvf} \left( \hc(s) \right)  \text{d}s
\end{align*}
According to Proposition \ref{Proplimith0}, the first integrand is bounded for $s$ small, but the second one explodes for small times since (cf.\ \eqref{eq:equequ})
 \begin{align*}
\left\lvert \sin\left( \theta(s) \right)c_{\rvf,J\rvf}^{J\rvf} \left( \hc(s) \right)  \right\rvert &\geqslant \left\lvert  3\frac{\sin\left( \theta(s) \right)}{\delta\left( \hc(s) \right)} \right\rvert \geqslant  \left\lvert  3\frac{\sin\left( \theta(s) \right)}{s} \right\rvert
\end{align*}
and  $\sin\left( \theta(s) \right)$ converges to a non-zero limit $\beta$ for $s\to 0$.
In both cases we find a contradiction and the statement is proved.
\end{proof}

Next we are ready to prove the regularity up to order one:
\begin{prop}\label{regularitythetazetaorder12}
The function $\theta:I\to S^{1}$, extended continuously by $\theta(0)=0$, is of class $C^{1}$ and $\theta'(0)=0$. 
 \end{prop}
\begin{proof} We study what happens for positive times. The result for negative times can be obtained similarly by reversing time. By Lemma \ref{lemmadiffequation}, for $t>0$ small enough
\begin{align}\label{Secondappeardiffequation}
 \left\lvert \sin\left( \theta(t) \right) \right\rvert \leqslant t \frac{M(t) }{3}.
\end{align}
Since $M(t)$ is bounded for $t\to 0$ by Proposition \ref{Proplimith0}, \eqref{Secondappeardiffequation} implies when $t\rightarrow 0$
\begin{align*}
\sin\left( \theta(t) \right)\longrightarrow0.
\end{align*}
We deduce that $\cos\left( \theta(t) \right)\to \alpha$ for $t\to 0$, with $\alpha=\pm 1$. Applying Lemma \ref{thetanotnegative}, we have $\alpha=1$. Hence when $t\rightarrow 0$
 \begin{align}\label{thetazetatendszero}
\theta(t)\longrightarrow 0.
\end{align}
Let us then extend $\theta$ by continuity defining 
$\theta(0)=0$, and let us prove that $\theta$ is indeed $C^{1}$. Notice that this implies
\begin{align}\label{deltaequivt}
\delta\left( \hc(t) \right)=\int_0^t \cos\left( \theta(s) \right)\text{d}s.
\end{align}
Combining Proposition~\ref{Proplimith0} and \eqref{thetazetatendszero}, one obtains that $M(t)\to 0$ for $t\to 0$. Hence for $t\to 0$ 
\begin{align}\label{sintovertgoeszero}
\left\lvert \frac{\sin( \theta(t) )}{t} \right\rvert \leqslant \frac{M(t) }{3}\longrightarrow 0.
\end{align}
which means that the function $\theta$ is differentiable at time zero and that
$\theta'(0)=0.$

To show that $\theta$ is $C^1$, we rewrite the differential equation satisfied by $\theta$ as :
\begin{align*}
 {\theta}'_{\hc}(t)&=\big(h_{\hc}(t)-\cos(\theta(t))\rd ( \hc(t) ) \big) +\frac{\sin( \theta(t) )}{t}\frac{t}{\delta( \hc(t) )}\delta( \hc(t) )c_{\rvf,J\rvf}^{J\rvf} ( \hc(t) ).
\end{align*}
By applying Propositions \ref{Proplimith0} - \ref{asymptoticsLieprojected}, and by using \eqref{thetazetatendszero}, \eqref{deltaequivt} and \eqref{sintovertgoeszero}, we deduce that for $t\to 0$ one has 
$ \theta'(t)\to 0 =\theta'(0)$, hence $\theta$ is $C^{1}$.
\end{proof}
Notice that along the lines of the proof we obtained the following well known fact about the metric speed of a horizontal curve. 
\begin{cor} For a smooth horizontal curve $\hc:I\to M$  with unit speed we have 
 \begin{align*}
 \lim_{t\to 0}\frac{d_{\SR}\left( \hc(t), \hc(0) \right)}{|t|}=1.
 \end{align*}
\end{cor}

\subsection{Second order} To prove second order regularity for $\theta$ we need to reformulate the identity characterizing $ k_{\hc}$ given in Proposition \ref{proplinkhktheta}.
\begin{lem}
We have for every $t\neq0$  
\begin{align}\label{kzetathetasimplifiedversion}
 k_{\hc}(t)=\theta''(t) + \frac{4\theta'(t)}{t} + \frac{2\theta(t)}{t^2} +r(t)
 \end{align}
 where $r(t)\to 0$ for $t\to 0$.

\end{lem}

\begin{proof} For the sake of simplicity, we focus on $t>0$. 
 We recall that from Proposition \ref{proplinkhktheta} the curvature $k_{\hc}(t)$ can be expressed as:
 \begin{align*}\label{ktheta}
  k_{\hc}(t)&=-\eta \left( \hc'(t) \right) +\theta''(t) - {\theta}'_{\hc}(t) \left( \sin\left( \theta(t) \right)\rd\left( \hc(t) \right)  +\cos\left( \theta(t) \right)c_{\rvf,J\rvf}^{J\rvf}\left( \hc(t) \right)   \right)\\
&\quad +\cos^2\left( \theta(t) \right) \left(  \cos \left( 2\theta(t) \right) \eta \left( \hc'(t) \right) +\sin \left( 2\theta(t) \right) \iota \left( \hc'(t) \right) \right) \\
& \quad +\sin\left( 2\theta(t) \right)  \left( - \sin \left( 2\theta(t) \right) \eta \left( \hc'(t) \right)
+\cos \left( 2\theta(t) \right) \iota \left( \hc'(t) \right) -\frac{1}{2}c_{\rvf,X_0}^{J\rvf} \right)\\
& \quad -\frac{\sin\left( 2\theta(t) \right)}{2} \left( \rvf c_{\rvf,J\rvf}^{J\rvf} \right)\left( \hc(t) \right)
-\sin^2\left( \theta(t) \right) \left( J\rvf c_{\rvf,J\rvf}^{J\rvf} \right)\left( \hc(t) \right).
 \end{align*}
Let us rewrite the three quantities
\begin{gather}
 -\theta'(t)\cos\left( \theta(t) \right)c_{\rvf,J\rvf}^{J\rvf}\left( \hc(t) \right), \\ 
  -\frac{\sin\left( 2\theta(t) \right)}{2}\left( c_{\rvf,X_0}^{J\rvf}\left( \theta(t) \right) +\left( \rvf c_{\rvf,J\rvf}^{J\rvf} \right)\left( \hc(t) \right) \right),\\
   -\sin^2\left( \theta(t) \right) \left( J\rvf c_{\rvf,J\rvf}^{J\rvf} \right)\left( \hc(t) \right),
\end{gather}
 as follows, respectively
\begin{gather}
 \frac{4\theta'(t)}{t}  + \frac{\theta'(t)}{t}\left( \frac{-t}{\delta\left( \hc(t) \right)}\cos\left( \theta(t) \right)\delta\left( \hc(t) \right)c_{\rvf,J\rvf}^{J\rvf}\left( \hc(t) \right) -4 \right),\\
 \frac{2\theta(t)}{t^2} + \frac{\theta(t)}{t^2}\left(\frac{-t^2}{\delta^2\left( \hc(t) \right)}\frac{\sin\left( 2\theta(t) \right)}{2\theta(t)}\delta^2\left( \hc(t) \right)\left( c_{\rvf,X_0}^{J\rvf} +\left( \rvf c_{\rvf,J\rvf}^{J\rvf} \right) \right)\left( \hc(t) \right) -2\right),\\
 -\left(\frac{\sin\left( \theta(t) \right)}{t}\right)^2\frac{t^2}{\delta^2\left( \hc(t) \right)}\delta^2\left( \hc(t) \right) \left( J\rvf c_{\rvf,J\rvf}^{J\rvf} \right)\left( \hc(t) \right),
\end{gather}
%
%
%
%
Using the Taylor expansion, together with the asymptotics of Propositions \ref{asymptoticsLieprojected} and \ref{regularitythetazetaorder12}, we obtain \eqref{kzetathetasimplifiedversion} with $r(t)$ which tends to zero when $t\to 0$.
\end{proof}
\begin{prop}\label{thetaC2}
The function $\theta:I\to S^{1}$, extended continuously by $\theta(0)=0$ is of class $C^{2}$, with  $\theta'(0)=0$ and 
$
   \theta''(0) =k_{\hc}(0)/6
$.
\end{prop}

\begin{proof}
Let us define, $f:I\setminus\{ 0 \}\to \mathbb R$ by
\begin{align}\label{notationmathfrakF}
f(t):=\theta'(t)+2\frac{\theta(t)}{t}-\frac{k_{\hc}(0)t}{3}.
\end{align}
According to Proposition \ref{regularitythetazetaorder12},  $f(t)\to 0$ for $t\to 0$.
We can rewrite \eqref{kzetathetasimplifiedversion} as follows 
\begin{align}\label{diffmathfrak}
\lim_{t\to 0} \left(f'(t)+2\frac{f(t)}{t}\right)= 0
\end{align}
Thanks to Lemma~\ref{l:ultimo} this implies that $f$ is of class $C^{1}$ on $I$ and  
$f'(0)=0$. 
Differentiating \eqref{notationmathfrakF}, this means that
\begin{align}\label{eq:thislim}
\lim_{t\to 0} \left(\theta''(t) +2\frac{\theta'(t)}{t}-2\frac{\theta(t)}{t^2}-\frac{k_{\hc}(0)}{3}\right)=0.
\end{align}
But if we sum \eqref{eq:thislim} with \eqref{kzetathetasimplifiedversion} (recall also Proposition~\ref{regularitythetazetaorder12}), we deduce that the function $g:I\setminus\{ 0 \}\to \mathbb R$ defined by
 $$g(t)=\theta'( t ) -\frac{k_{\hc}(0)}{6} t$$
satisfies the relations
\begin{align*}
\lim_{t\to 0} g(t) =0,\qquad \lim_{t\to 0} \left( g'(t) +3\frac{g(t)}{t}\right)=0,
\end{align*}
Applying now Lemma~\ref{l:ultimo} to $g$, we have that $g$ is of class $C^{1}$ on $I$ and  
$g'(0)=0$.
This proves that $\theta$ is $\mathcal{C}^2$ on $I$ and that 
$\theta''(0)=k_{\hc}(0)/6$, as required.
\end{proof}
In the previous proof we used twice the following lemma.
\begin{lem}\label{l:ultimo} Let $\varphi:I\setminus\{0\}\to \mathbb{R}$ be a $C^{1}$ function such that for some $\alpha>0$
\begin{align}\label{eq:rrr}
\lim_{t\to 0}\varphi(t) =0,\qquad \lim_{t\to 0} \left(\varphi'(t) +\alpha \frac{\varphi(t)}{t}\right)=0,
\end{align}
Then  $\varphi$ is of class $C^{1}$ on $I$ and $\varphi'(0) =0$.
\end{lem}
\begin{proof} It is sufficient to prove that  
\begin{equation}\label{eq:lim0}
\lim_{t\to 0}\frac{ \varphi(t)}{t}=0
\end{equation}
Indeed this implies that $\varphi$ is differentiable at zero and $\varphi'(0)=0$. Moreover from \eqref{eq:rrr} we get $\lim_{t\to 0}\varphi'(t) =0$, i.e., $\varphi$ is of class $C^{1}$ on $I$. 
If \eqref{eq:lim0} is not true, then there exists $t_0\neq 0$ arbitrarily small such that
\begin{align}\label{eq:SIE}
\left\lvert \alpha\frac{ \varphi(t_0)}{t_0}\right\rvert > \varepsilon.
\end{align}
We can assume $t_{0}>0$, the case $t_{0}<0$ being similar. By the second identity in \eqref{eq:rrr} we can choose $t_{0}$ such that for $0< t\leq t_{0}$
\begin{align}\label{eq:SIE1}
\left\lvert \varphi'(t)+\alpha \frac{\varphi(t)}{t}\right\rvert \leqslant \varepsilon.
\end{align}
By combining \eqref{eq:SIE} and \eqref{eq:SIE1}, we obtain that $\varphi(t_0)/t_0$ and $\varphi'(t_0)$ have opposite signs which implies, since $t_0>0$, that $\left(\varphi^2\right)'(t_0)\leqslant 0$. As a consequence
\begin{align*}
\frac{\text{d}}{\text{d}t}\bigg|_{t=t_0} \left\lvert \alpha\frac{\varphi(t)}{t}\right\rvert \leqslant 0.
\end{align*}
Therefore \eqref{eq:SIE} holds for every $t\in(0,t_0]$, and $\left(\varphi^2\right)'(t)\leqslant 0$ for every $t\in(0,t_0]$. But $\lim_{t\to 0}\varphi(t)=0$, hence $\varphi$ vanishes identically. This is in contradiction with \eqref{eq:SIE}. 
\end{proof}
\subsection{Proof of Theorem \ref{t:main}}
Thanks to Proposition~\ref{thetaC2}, the assumptions of Proposition \ref{PropTaylorTheta} are satisfied with the additional property that 
 \begin{align*}
   \theta''(0) &=\frac{k_{\hc}(0)}{6}.
  \end{align*}
%


\appendix

\section{Jacobi fields and asymptotics of the Lie brackets} \label{s:appjac}
In what follows we discuss asymptotics of sub-Riemannian Jacobi fields. In this appendix we give a self-contained presentation to prove Proposition~\ref{asymptoticsLieprojected}, but we refer \cite{agrachev2018curvature} (see also \cite{BRconnection} for a survey) for more general results, which contain in particular Lemma~\ref{LJ0P} presented below.

We denote   $S_{p}^{*}M:= H^{-1}(1/2) \cap T^*_{p}\M $, and we set
 \begin{align}
  \mathfrak{S}:=S_{p}^{*}M\bigcup \left\lbrace \mathrm{d}_{q}\delta \mid q\in \Sigma_p \right\rbrace.
 \end{align}
 where we recall $\mathrm{d}_{q}\delta$ denotes the differential of $\delta$ at $q$.
We can interpret $\mathfrak{S}$ as the union of the integral lines of the Hamiltonian flow that are the lifts to $T^*\M$ of geodesics leaving from $p$ that are parametrized by arc length and that have not yet reached their cut time. 
\begin{prop}\label{mapmathfrakS} Let $\Phi_{\vec H}^{t}$ denote the flow of $\vec H$. The map
\begin{align*}
 F :  \mathfrak{S}\to F( \mathfrak{S} ) \subset S^*_{p}\M  \times \mathbb{R},\qquad 
  F(\lambda)= \left( \Phi_{-\vec{H}}^{\delta(\pi(\lambda))}( \lambda ), \delta( \pi( \lambda ) ) \right),
\end{align*}
is a diffeomorphism whose inverse is :
\begin{align*}
F^{-1} : F( \mathfrak{S} ) \to \mathfrak{S},\qquad 
F^{-1}( \xi, \overline{\delta} )= \Phi_{\vec{H}}^{\overline{\delta}}\left( \xi \right).
\end{align*}
\end{prop}
We can see $( \xi,\overline{\delta})$ as coordinates on the set $\mathfrak{S}$. The function $\delta$ is thereby transported from its initial domain $\Sigma_p$ to   $\mathfrak{S}$, since $\overline{\delta}=\delta\circ\pi$.
In the coordinates $( \xi,\overline{\delta})$ 
 \begin{align}\label{partialdeltabar}
 \frac{\partial}{\partial \overline{\delta}} = \vec{H},\qquad 
\pi_* \circ \vec{H} = \rvf\circ\pi.
\end{align}
where $\pi_{*}$ denotes the differential of $\pi:TM\to M$.

\begin{rem}
As a consequence of  definitions 
 \eqref{defliftedfield} and \eqref{defpartialhXi}, for every orthonormal frame of the distribution $\{ X_1,X_2 \}$ and $ i,j,k = 0,1,2$ we have
\begin{align}
\left[ \overline{X}_i,\overline{X}_j \right]=\overline{\left[ X_i, X_j \right]},
\end{align}
which in turn implies
\begin{align}
c_{\overline{X}_i,\overline{X}_j}^{\overline{X}_k}= \overline{c}_{X_i,X_j}^{X_k},\qquad 
\left[ \frac{\partial}{\partial h_{X_i}} ,\overline{X}_j \right] = \left[ \frac{\partial}{\partial h_{X_i}} , \frac{\partial}{\partial h_{X_j}} \right]=0.
\end{align}
\end{rem}

We can now introduce Jacobi fields.\footnote{This is a vector field on $T^{*}M$, along the lift of a geodesic. To recover the classical notion of Jacobi field on $M$, one should consider the projection $J=\pi_{*}(\mathcal{J})$ onto $M$.} 
\begin{de}
A vector field $\mathcal{J}$ defined along  an integral line $\overline{\geo}:I \rightarrow T^*\M$ of the Hamiltonian field is said to be a \emph{Jacobi field} if the Lie derivative $\mathcal{L}_{\vec{H}}{\mathcal{J}}=0$ along $\overline{\geo}$.
\end{de}
We need the following result.
\begin{lem} \label{LJ0P}
 Let us consider $\{ X_1,X_2 \}$  an orthonormal frame of the distribution. 
 There exist two smooth vector fields $\mathcal{J}^{\perp} : \mathfrak{S}\rightarrow T\mathfrak{S}$ and $\mathcal{J}^{0} : \mathfrak{S}\rightarrow T\mathfrak{S}$
\begin{align}\label{eq:LJ0P}
 \mathcal{J}^i = \alpha^i \overline{X}_1 +\beta^i \overline{X}_2 +\sigma^i \overline{X}_0 + j^i_1 \frac{\partial}{\partial h_{X_1}} + j^i_2 \frac{\partial}{\partial h_{X_2}} +  j^i_0 \frac{\partial}{\partial h_{X_0}},\quad i\in\left\{ \perp, 0 \right\},
\end{align}
 that satisfy for $i\in\left\{ \perp, 0 \right\}$ the following conditions
\begin{itemize} 
\item[(i)]  $\mathcal{J}^i$ is a Jacobi field, i.e., $[\mathcal{J}^i,\vec{H}]=0$.
\item[(ii)] for every $\xi\in S^*_p\M$ we have
 $\pi_{*}\left(\mathcal{J}^i(\xi)\right) =0$ ($\mathcal{J}^i$ is vertical at zero) and
\begin{align}
 \sigma^{\perp}\circ F^{-1}\left( \xi,\overline{\delta} \right)\sim \frac{\overline{\delta}^2}{2}\quad \text{ and }\quad \sigma^{0}\circ F^{-1}\left( \xi,\overline{\delta} \right)\sim -\frac{\overline{\delta}^3}{6}.
\end{align}
\end{itemize}
Moreover, the functions $\sigma^i$ are smooth and are independent of the choice of $\{ X_1 , X_2 \}$.
 \end{lem}
\begin{proof}
 By combining the expression of $\vec{H}$ given by \eqref{expressionHamiltonian} and that of $\mathcal{J}^i$, we can reformulate the condition $[\mathcal{J}^i,\vec{H}]=0$ by decomposing it on the frame $$\left( \overline{X}_0, \overline{X}_1,\overline{X}_2,\frac{\partial}{\partial h_{X_0}}, \frac{\partial}{\partial h_{X_1}},\frac{\partial}{\partial h_{X_2}} \right).$$ The corresponding system of differential equation is given by:
 \begin{align}\label{systemJacobi}
\left\lbrace
\begin{array}{rcl}
\vec{H}\sigma^i &=& h_{X_2}\alpha^i -h_{X_1}\beta^i\\
\vec{H}\alpha^i &=& \left( h_{X_2}\alpha^i -h_{X_1}\beta^i \right) \overline{c}_{1,2}^{1} +h_{X_1} \sigma^i \eta\left( X_1 \right) +h_{X_2} \sigma^i \overline{c}_{0,2}^{1} +j_1^i\\
\vec{H}\beta^i &=& \left( h_{X_2}\alpha^i -h_{X_1}\beta^i \right) \overline{c}_{1,2}^{2} +h_{X_2} \sigma^i \eta\left( X_2 \right) +h_{X_1} \sigma^i \overline{c}_{0,1}^{2} +j_2^i\\
\vec{H}j_1^i &=&-\sum_{k}\left( h_{X_2} h_{X_k}\left(\alpha^i X_1 + \beta^i X_2 +\sigma^i X_0 \right)\overline{c}_{1,2}^{k}\right.\\
&&\qquad \qquad\qquad\qquad\quad\left. +\overline{c}_{1,2}^{k} \left(j_2^i h_{X_k} +h_{X_2} j_k^i \right)\right)\\
\vec{H}j_2^i &=&-\sum_{k}\left( h_{X_1} h_{X_k}\left(\alpha^i X_1 + \beta^i X_2 +\sigma^i X_0 \right)\overline{c}_{2,1}^{k}\right.\\
&&\qquad \qquad\qquad\qquad\quad\left. +\overline{c}_{2,1}^{k} \left(j_1^i h_{X_k} +h_{X_1} j_k^i \right)\right)\\
\vec{H} j_0^i &=&\sum_{k, \ell\neq 0}  \left( h_{X_\ell} h_{X_k}\left(\alpha^i X_1 + \beta^i X_2 +\sigma^i X_0 \right)\overline{c}_{0,\ell}^{k}\right.\\
&&\qquad \qquad\qquad\qquad\qquad\left. +\overline{c}_{0,j}^{k} \left(j_\ell^i h_{X_k} +h_{X_\ell} j_k^i \right)\right).
\end{array}\right.
\end{align}

In order to define the vector fields $\mathcal{J}^{\perp}$ and $\mathcal{J}^{0}$, it is then sufficient to define their values on $S^*_{p}\M$, the values on the whole space $F(\mathfrak{S})$ then following from the differential equation \eqref{systemJacobi}.
We define on $S^*_{p}\M$
\begin{align*}
\mathcal{J}^{\perp}\big|_{S^*_{p}\M}=h_{X_2}\frac{\partial}{\partial h_{X_1}}-h_{X_1}\frac{\partial}{\partial h_{X_2}},\qquad \mathcal{J}^{0}\big|_{S^*_{p}\M}=\frac{\partial}{\partial h_{X_0}}.
\end{align*}

We now use \eqref{systemJacobi} to establish the asymptotics of $\sigma^0$ and $\sigma^\perp$ by computing the derivatives and evaluating  at zero. We find out that for any $\xi$ in $S^*_{p}\M$
\begin{align*}
\sigma^{i}\left( \xi \right)&=0, \\
\vec{H}( \sigma^{i} )\left( \xi \right) &=  h_{X_2}(\xi)\alpha^{i}(\xi) -h_{X_1}(\xi)\beta^{i}(\xi) =0, \\
\vec{H}^2( \sigma^{i} )\left( \xi \right) &=  h_{X_2}(\xi) \vec{H}( \alpha^{i} )(\xi) - h_{X_1}(\xi) \vec{H}( \beta^{i} )(\xi)\\
&=h_{X_2}(\xi) j_1^i(\xi) - h_{X_1}(\xi) j_2^i(\xi),
\end{align*}
where we used $\alpha^{i}(\xi)=\beta^{i}(\xi)=\sigma^{i}(\xi)=0$ at $S^*_{p}\M$. Using $\xi\in H^{-1}\left( 1/2 \right)$, we have
\begin{align*}
\vec{H}^2( \sigma^{\perp} )\left( \xi \right)&= h_{X_2}^2(\xi) +h_{X_1}^2(\xi)= 1,\qquad  \vec{H}^2( \sigma^{0} )\left( \xi \right)=0.
\end{align*}
Furthermore, using again $\xi\in H^{-1}\left( 1/2 \right)$
\begin{align*}
\vec{H}^3( \sigma^{0} )\left( \xi \right)&= h_{X_2}(\xi) \vec{H}( j_1^i )(\xi) - h_{X_1}(\xi) \vec{H} ( j_2^i )(\xi)\\
&= - h_{X_2}^2(\xi) -h_{X_1}^2(\xi)= - 1.
\end{align*}
Now, thanks to the first identity in \eqref{partialdeltabar}
\begin{align*}
\vec{H}^n ( \sigma^i )\left( \xi \right) ={\frac{\partial^n}{\partial \overline{\delta}^n}}\bigg|_{\overline{\delta} =0}\sigma^i \circ F^{-1} \left( \xi,\overline{\delta} \right),
\end{align*}
which proves the asymptotics in (ii). The fact that the functions $\sigma^i$ are smooth and independent of the choice of $\left( X_1,X_2 \right)$ is due to the identity $
 \sigma^i=\omega\circ\text{d}\pi\left( \mathcal{J}^i \right)$.
\end{proof}

\begin{lem}\label{LFGJG}
We have
$  \mathfrak{S} \setminus S^*_p\M  \subset h_{\rvf}^{-1}\left( 1 \right) \cap h_{J\rvf}^{-1}\left( 0 \right)$.
\end{lem}
\begin{proof}
 Let us consider any covector in $\mathfrak{S}\setminus S^*_p\M$. It can be written as
$
 \text{d}\delta_q
$
for a certain $q$ in $\Sigma_p$ by definition of $\mathfrak{S}$.
Now if we choose $\left( \rvf,J\rvf\right)$ as a frame of the distribution, by using \eqref{defcoordinateh},
we can write thanks to Lemma \ref{propgraddelta}
\begin{equation*} 
h_{\rvf}\left( \text{d}\delta_q \right) =\text{d}\delta_q \left( \rvf \right)=1,
\qquad
h_{J\rvf}\left( \text{d}\delta_q \right) =\text{d}\delta_q \left( J\rvf \right)=0.  \qedhere
\end{equation*}
\end{proof}

We are now able to compute the asymptotics of the Lie brackets of the elements of the  frame $\left( \rvf,J\rvf, X_0 \right)$.

\begin{prop} \label{1aLie} The quantities $\overline{\delta} c_{\rvf,J\rvf}^{J\rvf}$ and $\overline{\delta}^2 c_{\rvf,X_0}^{J\rvf}$ (a priori defined on $\mathfrak{S}\setminus S^*_p\M $)  can be smoothly extended to $\mathfrak{S}$ and are respectively equal to $-4$ and $-6$ over $S^*_p\M$.
\end{prop}
\begin{proof}
Let write the vector fields $\mathcal{J}^0$ and $\mathcal{J}^{\perp}$ introduced in Lemma \ref{LJ0P}over $\mathfrak{S}\setminus T^*_p\M$ with respect to the orthonormal frame $\Gamma,J\Gamma$: 
\begin{align*}
 \mathcal{J}^i = \alpha^i \overline{\rvf} +\beta^i \overline{J\rvf} +\sigma^i \overline{X}_0 + j^i_1 \frac{\partial}{\partial h_{\rvf}} + j^i_2 \frac{\partial}{\partial h_{J\rvf}} +  j^i_0 \frac{\partial}{\partial h_{X_0}},
\end{align*}
Since $\mathfrak{S}\setminus T^*_p\M$ is contained in $h_{\rvf}^{-1}\left( 1 \right) \cap h_{J\rvf}^{-1}\left( 0 \right)$ (by Lemma \ref{LFGJG}) we have
$ j^i_1 =j^i_2=0$,
and the first and the third equation of \eqref{systemJacobi} can be combined as
\begin{align*}
 0&=\vec{H}^2 \sigma^i + \overline{c}_{\rvf,J\rvf}^{J\rvf}\vec{H} \sigma^i +\overline{c}_{X_0,\rvf}^{J\rvf}\sigma^i.
\end{align*}
Since this last equation is satisfied by $\sigma^{0}$ and $\sigma^{\perp}$ we find out that 
\begin{align}\label{SystemdeltaLiebrackets}
\begin{pmatrix}
 -\vec{H}^2 \sigma^{\perp}\\
 -\frac{\vec{H}^2 \sigma^{0}}{\overline{\delta}} 
\end{pmatrix}
=\begin{pmatrix}
 \frac{\vec{H} \sigma^{\perp}}{\overline{\delta}} & \frac{\sigma^{\perp}}{\overline{\delta}^2}\\
 \frac{\vec{H} \sigma^{0}}{\overline{\delta}^2} & \frac{\sigma^{0}}{\overline{\delta}^3}
\end{pmatrix}
\begin{pmatrix}
 \overline{\delta} \overline{c}_{\rvf,J\rvf}^{J\rvf} \\
 \overline{\delta}^2 \overline{c}_{X_0,\rvf}^{J\rvf}
\end{pmatrix}.
\end{align}
 The matrix of this system as well as its left hand side are smooth over $\mathfrak{S}$ when $\overline{\delta}$ goes to zero  by applying  Lemma \ref{LemmaRegularity} to the asymptotics given in Lemma \ref{LJ0P}  (we use the first identity in \eqref{partialdeltabar}), and
 \begin{align*}
\begin{pmatrix}
 -\vec{H}^2 \sigma^{\perp}\\
 -\frac{\vec{H}^2 \sigma^{0}}{\overline{\delta}} 
\end{pmatrix} \to   \begin{pmatrix}
   -1\\
   1
  \end{pmatrix} ,\qquad \begin{pmatrix}
 \frac{\vec{H} \sigma^{\perp}}{\overline{\delta}} & \frac{\sigma^{\perp}}{\overline{\delta}^2}\\
 \frac{\vec{H} \sigma^{0}}{\overline{\delta}^2} & \frac{\sigma^{0}}{\overline{\delta}^3}
\end{pmatrix} \to
\begin{pmatrix}
 1 & \frac{1}{2}\\
 -\frac{1}{2} & -\frac{1}{6}
\end{pmatrix}. 
 \end{align*}
Inverting  \eqref{SystemdeltaLiebrackets} we obtain that the functions $\overline{\delta} \overline{c}_{\rvf,J\rvf}^{J\rvf}$ and $\overline{\delta}^2 \overline{c}_{X_0,\rvf}^{J\rvf}$ that were a priori defined on $\mathfrak{S}\setminus S^*_p\M $, can in fact be smoothly extended to the domain $\mathfrak{S}$. Taking then the limit as $\overline{\delta}$ goes to zero, we find the values of $\overline{\delta} \overline{c}_{\rvf,J\rvf}^{J\rvf}$ and $\overline{\delta}^2 \overline{c}_{X_0,\rvf}^{J\rvf}$ on the set $\overline{\delta}^{-1}( 0 )=S^*_p\M$.
\end{proof}
 
 We  obtain similar results for $\overline{\delta}^2 \vec{H} \overline{c}_{\rvf,J\rvf}^{J\rvf}$ and $\overline{\delta}^2 \overline{J\rvf} \overline{c}_{\rvf,J\rvf}^{J\rvf}$. 
\begin{prop}\label{2aLie}
 The function $\overline{\delta}^2 \vec{H} \overline{c}_{\rvf,J\rvf}^{J\rvf}$, a priori defined on $\mathfrak{S}\setminus S^*_p\M $,  can be extended to a smooth function on  $\mathfrak{S}$ and its evaluation is equal to $4$ on $S^*_p\M$.
\end{prop}
\begin{proof}
 We know from Proposition \ref{1aLie} that $\overline{\delta} \overline{c}_{\rvf,J\rvf}^{J\rvf}$ can be extended to a smooth function on $\mathfrak{S}$ that is equal to $-4$ on $\overline{\delta}^{-1}\left( 0 \right)$. Since $\vec{H}$ is also smooth, we can write
\begin{align*}
 \overline{\delta} \vec{H} \left(\overline{\delta} \overline{c}_{\rvf,J\rvf}^{J\rvf}  \right)&=\delta\left( \vec{H}\overline{\delta} \right) \overline{c}_{\rvf,J\rvf}^{J\rvf} +\overline{\delta}^2 \left( \vec{H}\overline{c}_{\rvf,J\rvf}^{J\rvf} \right).
\end{align*}
So, recalling that $\vec{H}\overline{\delta} $, we have that
\begin{align*}
 \overline{\delta}^2 \left( \vec{H}\overline{c}_{\rvf,J\rvf}^{J\rvf} \right) =\overline{\delta} \vec{H} \left(\overline{\delta} \overline{c}_{\rvf,J\rvf}^{J\rvf}  \right) -\overline{\delta} \overline{c}_{\rvf,J\rvf}^{J\rvf}
\end{align*}
has a smooth extension on $\mathfrak{S}$ that is equal to $4$ on $\overline{\delta}^{-1}\left( 0 \right)=S^*_p \M$.
\end{proof}

\begin{prop}\label{3aLie}
 The function $\overline{\delta}^2 \overline{J\rvf} \overline{c}_{\rvf,J\rvf}^{J\rvf}$ that is a priori defined on $\mathfrak{S}\setminus S^*_p\M$ can be extended to a smooth function on the domain $\mathfrak{S}$.
\end{prop}
\begin{proof} Let us consider the fields $\mathcal{J}^0$ and $\mathcal{J}^{\perp}$ over $\mathfrak{S}$ that we introduced in Lemma \ref{LJ0P}.
We start by proving the following claim: 
for every $\xi$ in $\mathfrak{S}$ the vector $V(\xi)$ defined by \eqref{definitionxi} is colinear to $J\rvf$.
\begin{align}\label{definitionxi}
 V(\xi):=\text{d}\pi \left( \sigma^{\perp}(\xi) \mathcal{J}^{0}(\xi) - \sigma^{0}(\xi) \mathcal{J}^{\perp}(\xi)   \right). 
\end{align}
First notice that $V(\xi)$ belongs to the distribution,  since its component along the Reeb vector field $X_0$ is zero thanks to \eqref{eq:LJ0P}.
Let us then prove that it is orthogonal to the gradient of $\delta$.
Indeed for $i\in \{0,\perp\}$, by definition of Jacobi field one has
 \begin{align*}
  0&=\left[ \vec{H},\mathcal{J}^i \right] \overline{\delta}= \vec{H}\mathcal{J}^i \overline{\delta} - \mathcal{J}^i (\vec{H}\overline{\delta})=\vec{H}\mathcal{J}^i \overline{\delta},
 \end{align*}
where we used that $\vec{H}\overline{\delta}=1$. So $\mathcal{J}^i \overline{\delta} $ is constant on the integral lines of $\vec{H}$. But on $S^*_p\M$, the function $\mathcal{J}^i \overline{\delta} $ is equal to zero.
Therefore 
 $\mathcal{J}^i \overline{\delta} =0$
on $\mathfrak{S}$. 
Thus d$\pi\left( \mathcal{J}^i\left( \xi \right) \right)$ belongs to the kernel of d$\delta$.
The vector $V(\xi)$ is a linear combination of vectors in the kernel of d$\delta$, hence it is also in the kernel of d$\delta$. By Lemma \ref{propgraddelta}, this means that $V(\xi)$ is colinear to $J\rvf$ and the claim is proved.

\medskip

Let us define
$  b:\mathfrak{S} \longrightarrow \mathbb{R}$
such that for any $\xi$ in $\mathfrak{S}$,
\begin{align}\label{xiprojectJGamma}
 V(\xi)=b(\xi)J\rvf\left( \pi(\xi) \right),
\end{align}
in such a way that for all $\xi$ in $\mathfrak{S}\setminus T^*_p\M$
\begin{align*}
 \text{d}\pi \left( \frac{ \sigma^{\perp}(\xi) \mathcal{J}^{0}(\xi) - \sigma^{0}(\xi) \mathcal{J}^{\perp}(\xi) }{b(\xi)}  \right)=J\rvf\left( \pi(\xi) \right).
\end{align*}
Since $\overline{c}_{\rvf,J\rvf}^{J\rvf}$ is
constant on the fiber of $T^*\M$, we can replace $\overline{J\rvf}$ in the expression with a vector field that projects over $J\rvf$. Then we have
\begin{align} \label{FFJG}
\overline{\delta}^2 \overline{J\rvf}\overline{c}_{\rvf,J\rvf}^{J\rvf} &= \overline{\delta}^2 \left( \frac{ \sigma^{\perp} \mathcal{J}^{0} \overline{c}_{\rvf,J\rvf}^{J\rvf}  - \sigma^{0} \mathcal{J}^{\perp} \overline{c}_{\rvf,J\rvf}^{J\rvf} }{b}\right).
\end{align}
To prove that the right hand side of \eqref{FFJG} is smooth, we write the fields $\mathcal{J}^i$ over $\mathfrak{S}\setminus T^*_p\M $ as :
\begin{align*}
 \mathcal{J}^i = \alpha^i \overline{\rvf} +\beta^i \overline{J\rvf} +\sigma^i \overline{X}_0 + j^i_1 \frac{\partial}{\partial h_{\rvf}} + j^i_2 \frac{\partial}{\partial h_{J\rvf}} +  j^i_0 \frac{\partial}{\partial h_{X_0}}.
\end{align*}
Combining with \eqref{definitionxi} and \eqref{xiprojectJGamma}, we obtain that for any $\xi$ in $\mathfrak{S}\setminus T^*_p\M $,
\begin{align*}
b(\xi)=\sigma^{\perp}(\xi)\beta^0(\xi) - \sigma^{0}(\xi)\beta^{\perp}(\xi).
\end{align*}
Now thanks to Lemma \ref{LFGJG}, $h_\rvf=1$  and $h_{J\rvf}=0$  on $\mathfrak{S}\setminus T^*_p\M $, so the first equation of
\eqref{systemJacobi} becomes
$\vec{H}\sigma^i = -\beta^i$.
As a consequence, for every $\xi$ in $\mathfrak{S}\setminus  T^*_p\M $,
\begin{align}
 b(\xi) =\sigma^{0}(\xi)\vec{H}\sigma^{\perp}(\xi)-\sigma^{\perp}(\xi)\vec{H}\sigma^{0}(\xi).
\end{align}
Hence
\begin{align*}
\overline{\delta}^2 \overline{J\rvf}\overline{c}_{\rvf,J\rvf}^{J\rvf} &= \overline{\delta}^2   \left(\frac{ \sigma^{\perp} \mathcal{J}^{0} \overline{c}_{\rvf,J\rvf}^{J\rvf}  - \sigma^{0} \mathcal{J}^{\perp} \overline{c}_{\rvf,J\rvf}^{J\rvf} }{b}\right) \\
 &= \overline{\delta}  \left(\frac{ \sigma^{\perp} \mathcal{J}^{0} \left( \overline{\delta} \overline{c}_{\rvf,J\rvf}^{J\rvf} \right)  - \sigma^{0} \mathcal{J}^{\perp}\left( \overline{\delta} \overline{c}_{\rvf,J\rvf}^{J\rvf} \right) }{\sigma^{0}\vec{H}\sigma^{\perp}-\sigma^{\perp}\vec{H}\sigma^{0}} \right)
 \end{align*}
where we used that $\mathcal{J}^{i}\overline\delta=0$ for $i\in\{0,\perp\}$. By applying Proposition \ref{1aLie}, $\overline{\delta} \overline{c}_{\rvf,J\rvf}^{J\rvf}$ can be extended to a smooth function defined on $\mathfrak{S}$
and its value on $S^*_p \M$ is constant. Therefore, the functions $\mathcal{J}^{i} ( \overline{\delta} \overline{c}_{\rvf,J\rvf}^{J\rvf} )$ can be extended to smooth functions on $\mathfrak{S}$
that vanish at every point of $S^*_p \M$. We combine this with the smoothness and the asymptotics of the functions $\sigma^i$ that come from Lemma \ref{LJ0P}, and thanks to Lemma \ref{LemmaRegularity}, the function $\overline{\delta}^2 \overline{J\rvf}c_{\rvf,J\rvf}^{J\rvf}$ can be
extended to a smooth function on  $\mathfrak{S}$.
\end{proof}

\subsection{Proof of Proposition~\ref{asymptoticsLieprojected}}

Let us start by proving the first identity.
We consider the  lift of $\hc$ defined by
$\overline{\hc} : I\setminus\{0\} \to  \mathfrak{S}$, where $\overline \hc (t)= \text{d}\delta_{\hc(t)}$.
Recall that $\overline{\delta} =\delta\circ\pi$ and that $\overline{c}_{i,j}^{k}={c}_{i,j}^{k}\circ \pi$. Therefore,
\begin{align}\label{ASPR}
\delta\left( \hc(t) \right)c_{\rvf,J\rvf}^{J\rvf}\left( \hc(t) \right)=\overline{\delta}\left( \overline{\hc}(t) \right)\overline{c}_{\rvf,J\rvf}^{J\rvf}\left( \overline{\hc}(t) \right).
\end{align}
Recall moreover that $\overline{\hc}(t)=\text{d}\delta_{\hc(t)}$ is the evaluation at time $\delta(\hc(t))$ of the integral line of the Hamiltonian
flow $\overline{\geo}_{\hc(t)}$, that is a lift of the minimizing geodesic $\geo_{\hc(t)}$ parametrized by arc length joining $p$ to $\hc(t)$.
In particular, by  Proposition \ref{mapmathfrakS},
\begin{align}\label{zetaHF}
 \overline{\hc}(t)=F^{-1}\left( \overline{\geo}_{\hc(t)}(0) ,\delta\left( \hc(t) \right) \right).
\end{align}
Combining  Remark~\ref{ProphX0geod} with \eqref{eq:limith0no} (cf.\ proof of Proposition~\ref{Proplimith0}), one has
\begin{align}\label{limithX0gammabar}
  h_{X_0}\left( \overline{\geo}_{\hc(t)}(0) \right) \stackrel{t\rightarrow 0}{\longrightarrow}h_{\hc}(0).
\end{align}
Since $\geo_{\hc(t)}$ is parametrized by arc length, then  $\overline{\geo}_{\hc(t)}$ is contained in $H^{-1}(1/2)$, which implies for $\left( X_1,X_2 \right)$ any choice of 
orthonormal frame of the distribution
\begin{align}\label{gammabarinE1}
 h^{2}_{X_1}\left( \overline{\geo}_{\hc(t)}(0) \right) +h^{2}_{X_2}\left( \overline{\geo}_{\hc(t)}(0) \right) =1.
\end{align}
By combining \eqref{zetaHF}, \eqref{limithX0gammabar}, \eqref{gammabarinE1}  we obtain that for $t$ small enough (recall $
 \delta\left( \hc(t)\right)\to 0
$), $ \overline{\hc}(t)$ belongs to a compact subset $K$
of $\mathfrak{S}$. Thanks to Proposition \ref{1aLie}, the function $\overline{\delta}\overline{c}_{\rvf,J\rvf}^{J\rvf}$ is uniformly continuous on $K$.
Now since $\overline{\delta}( \overline{\hc}(t) )=\delta( \hc(t) )\to 0$ for $t\to 0$, and as $\overline{\delta}\overline{c}_{\rvf,J\rvf}^{J\rvf}$ is equal to
$-4$ on $\overline{\delta}^{-1}(0)=S^*_p\M$, we deduce from the uniform continuity 
\begin{align*}
\lim_{t\to 0} \overline{\delta}\left( \overline{\hc}(t) \right)\overline{c}_{\rvf,J\rvf}^{J\rvf}\left( \overline{\hc}(t) \right)=-4.
\end{align*}
which proves, thanks to \eqref{ASPR},  the first claim.

The other asymptotics follows from similar arguments, where we use Propositions \ref{2aLie} and \ref{3aLie} instead of Proposition \ref{1aLie} and where we replace \eqref{ASPR} by the  relations
\begin{align*}
 {\delta}^2\left( \hc(t) \right)c_{\rvf,X_0}^{J\rvf}\left( \hc(t) \right) &=  \overline{\delta}^2\left( \overline{\hc}(t) \right)\overline{c}_{\rvf,X_0}^{J\rvf}\left( \overline{\hc}(t) \right) ,\\
 {\delta}^2\left( \hc(t) \right)\rvf c_{\rvf,J\rvf}^{J\rvf}\left( \hc(t) \right) &= \overline{\delta}^2\left(  \overline{\hc}(t) \right)\vec{H} \overline{c}_{\rvf,J\rvf}^{J\rvf}\left( \overline{\hc}(t) \right), \\
 {\delta}^2\left( \hc(t) \right)J\rvf c_{\rvf,J\rvf}^{J\rvf}\left( \hc(t) \right) &=\overline{\delta}^2\left( \overline{\hc}(t) \right)\overline{J\rvf} \overline{c}_{\rvf,J\rvf}^{J\rvf}\left( \overline{\hc}(t) \right).
\end{align*}
which are proved as \eqref{ASPR}, using that $\rvf\circ\pi=\text{d}\pi\circ\vec{H}$ (cf.\ \eqref{partialdeltabar}) and 
$J\rvf\circ\pi=\text{d}\pi\circ \overline{J\rvf}$.

\bibliographystyle{alphabbrv}
\bibliography{biblio}

\begin{bibdiv}
\begin{biblist}

\bib{agrachev2012sub}{article}{
      author={Agrachev, Andrei},
      author={Barilari, Davide},
       title={Sub-{R}iemannian structures on 3d {L}ie groups},
        date={2012},
     journal={Journal of Dynamical and Control Systems},
      volume={18},
      number={1},
       pages={21\ndash 44},
}

\bib{ABB}{article}{
      author={Agrachev, Andrei},
      author={Barilari, Davide},
      author={Boscain, Ugo},
       title={A {C}omprehensive {I}ntroduction to {S}ub-{R}iemannian
  {G}eometry},
        date={2019},
     journal={Cambridge University Press},
         url={http://webusers.imj-prg.fr/~davide.barilari/Notes.php},
}

\bib{agrachev2017sub}{article}{
      author={Agrachev, Andrei},
      author={Barilari, Davide},
      author={Rizzi, Luca},
       title={Sub-{R}iemannian curvature in contact geometry},
        date={2017},
     journal={The Journal of Geometric Analysis},
      volume={27},
      number={1},
       pages={366\ndash 408},
}

\bib{agrachev2018curvature}{article}{
      author={Agrachev, A.},
      author={Barilari, D.},
      author={Rizzi, L.},
       title={Curvature: a variational approach},
        date={2018},
        ISSN={0065-9266},
     journal={Mem. Amer. Math. Soc.},
      volume={256},
      number={1225},
       pages={v+142},
         url={https://doi.org/10.1090/memo/1225},
      review={\MR{3852258}},
}

\bib{AG99}{article}{
      author={Agrachev, Andrei~A.},
      author={Gauthier, Jean-Paul~A.},
       title={On the {D}ido problem and plane isoperimetric problems},
        date={1999},
        ISSN={0167-8019},
     journal={Acta Appl. Math.},
      volume={57},
      number={3},
       pages={287\ndash 338},
         url={https://doi.org/10.1023/A:1006237201915},
      review={\MR{1722037}},
}

\bib{Agrasmoothness}{article}{
      author={Agrachev, A.},
       title={Any sub-{R}iemannian metric has points of smoothness},
        date={2009},
        ISSN={0869-5652},
     journal={Dokl. Akad. Nauk},
      volume={424},
      number={3},
       pages={295\ndash 298},
         url={http://dx.doi.org/10.1134/S106456240901013X},
}

\bib{agrachev1996exponential}{article}{
      author={Agrachev, AA},
       title={Exponential mappings for contact sub-{R}iemannian structures},
        date={1996},
     journal={Journal of dynamical and control systems},
      volume={2},
      number={3},
       pages={321\ndash 358},
}

\bib{BRconnection}{article}{
      author={Barilari, Davide},
      author={Rizzi, Luca},
       title={On {J}acobi fields and a canonical connection in sub-{R}iemannian
  geometry},
        date={2017},
        ISSN={0044-8753},
     journal={Arch. Math. (Brno)},
      volume={53},
      number={2},
       pages={77\ndash 92},
         url={https://doi.org/10.5817/AM2017-2-77},
      review={\MR{3672782}},
}

\bib{balogh2017intrinsic}{article}{
      author={Balogh, Zolt{\'a}n~M},
      author={Tyson, Jeremy~T},
      author={Vecchi, Eugenio},
       title={Intrinsic curvature of curves and surfaces and a
  {G}auss--{B}onnet theorem in the {H}eisenberg group},
        date={2017},
     journal={Mathematische Zeitschrift},
      volume={287},
      number={1-2},
       pages={1\ndash 38},
}

\bib{chiu2}{article}{
      author={Chiu, Hung-Lin},
      author={Feng, XiuHong},
      author={Huang, Yen-Chang},
       title={The differential geometry of curves in the {H}eisenberg groups},
        date={2018},
        ISSN={0926-2245},
     journal={Differential Geom. Appl.},
      volume={56},
       pages={161\ndash 172},
         url={https://doi.org/10.1016/j.difgeo.2017.11.009},
      review={\MR{3759360}},
}

\bib{chiu1}{article}{
      author={Chiu, Hung-Lin},
      author={Huang, Yen-Chang},
      author={Lai, Sin-Hua},
       title={An application of the moving frame method to integral geometry in
  the {H}eisenberg group},
        date={2017},
        ISSN={1815-0659},
     journal={SIGMA Symmetry Integrability Geom. Methods Appl.},
      volume={13},
       pages={Paper No. 097, 27},
         url={https://doi.org/10.3842/SIGMA.2017.097},
      review={\MR{3740335}},
}

\bib{chan}{article}{
      author={Chanillo, Sagun},
      author={Yang, Paul~C.},
       title={Isoperimetric inequalities \& volume comparison theorems on {CR}
  manifolds},
        date={2009},
        ISSN={0391-173X},
     journal={Ann. Sc. Norm. Super. Pisa Cl. Sci. (5)},
      volume={8},
      number={2},
       pages={279\ndash 307},
      review={\MR{2548248}},
}

\bib{diniz2016gauss}{article}{
      author={Diniz, Marcos~M},
      author={Veloso, Jos{\'e}~MM},
       title={Gauss-{B}onnet theorem in sub-{R}iemannian {H}eisenberg space $
  \mathbb{H}^1$},
        date={2016},
     journal={Journal of Dynamical and Control Systems},
      volume={22},
      number={4},
       pages={807\ndash 820},
}

\bib{el1996small}{article}{
      author={El-Alaoui, El-H~Ch},
      author={Gauthier, J-P},
      author={Kupka, I},
       title={Small sub-{R}iemannian balls onr 3},
        date={1996},
     journal={Journal of dynamical and Control Systems},
      volume={2},
      number={3},
       pages={359\ndash 421},
}

\bib{MK19}{article}{
      author={Kohli, Mathieu},
       title={A metric interpretation of the geodesic curvature in the
  {H}eisenberg group},
        date={2019Apr},
        ISSN={1573-8698},
     journal={Journal of Dynamical and Control Systems},
         url={https://doi.org/10.1007/s10883-019-09444-7},
}

\bib{MKphd}{article}{
      author={Kohli, Mathieu},
       title={On the geodesic curvature in sub-{R}iemannian geometry},
        date={2019},
     journal={PhD thesis},
}

\bib{rumin}{article}{
      author={Rumin, Michel},
       title={Formes diff\'erentielles sur les vari\'et\'es de contact},
        date={1994},
     journal={J. Differential Geom.},
      volume={39},
      number={2},
       pages={281\ndash 330},
         url={https://doi.org/10.4310/jdg/1214454873},
}

\bib{tanno89}{article}{
      author={Tanno, Shukichi},
       title={Variational problems on contact {R}iemannian manifolds},
        date={1989},
        ISSN={0002-9947},
     journal={Trans. Amer. Math. Soc.},
      volume={314},
      number={1},
       pages={349\ndash 379},
         url={http://dx.doi.org/10.2307/2001446},
      review={\MR{1000553 (90f:53071)}},
}

\end{biblist}
\end{bibdiv}

\end{document}